\documentclass[11pt,a4paper,reqno]{amsart}
\usepackage[english]{babel}
\usepackage[applemac]{inputenc}
\usepackage[T1]{fontenc}
\usepackage{palatino}
\usepackage{bbm}
\usepackage{cite}
\usepackage{float, verbatim}
\usepackage{amsmath}
\usepackage{amssymb}
\usepackage{amsthm}
\usepackage{amsfonts}
\usepackage{graphicx}
\usepackage{mathtools}
\usepackage{enumitem}
\usepackage[colorlinks = true, citecolor = black]{hyperref}
\usepackage{color}
\usepackage{tikz}
\usepackage{xcolor}
\usepackage{pgfplots}
\usepackage{caption}
\usepackage{subcaption}
\usepackage{enumerate}
\usepackage{autobreak}
\usepackage{url}

\newcommand{\Assouad}{\dim_{\mathrm{A}}}
\newcommand{\pAssouad}{\dim_{\mathrm{mA}}}
\newcommand{\ubox}{\overline{\dim_{\mathrm{B}}}}

\newcommand{\lbox}{\underline{\dim_{\mathrm{B}}}}
\newcommand{\boxd}{\dim_{\mathrm{B}}}

\newcommand{\Haus}{\dim_{\mathrm{H}}}

\newtheorem*{thm*}{Theorem}
\newtheorem*{defn*}{Definition}

\newtheorem{thm}{Theorem}[section]

\newtheorem{lma}[thm]{Lemma}
\newtheorem{cor}[thm]{Corollary}
\newtheorem{defn}[thm]{Definition}

\newtheorem{prop}[thm]{Proposition}
\newtheorem{conj}[thm]{Conjecture}
\newtheorem{rem}[thm]{Remark}

\begin{document}
	
	\title{An improvement on Furstenberg's intersection problem }
	
	\author{Han Yu}
	\address{Han Yu\\
		Department of Pure Mathematics and Mathematical Statistics\\University of Cambridge\\CB3 0WB \\ UK }
	\curraddr{}
	\email{hy351@maths.cam.ac.uk}
	\thanks{}
	
	
	
	\date{}
	
	\dedicatory{}
	
	\begin{abstract}
		In this paper, we study a problem posed by Furstenberg on intersections between $\times 2, \times 3$ invariant sets. We present here a direct geometrical counting argument to revisit a theorem of Wu and Shmerkin. This argument can be used to obtain further improvements. For example, we show that if $A_2,A_3\subset [0,1]$ are closed and $\times 2, \times 3$ invariant respectively, assuming that $\dim A_2+\dim A_3<1$ then $A_2\cap (uA_3+v)$ is sparse (defined in this paper) and has box dimension zero uniformly with respect to the real parameters $u,v$ such that $u$ and $u^{-1}$ are both bounded away from $0$.
	\end{abstract}
	
	\maketitle
	%\tableofcontents
	\allowdisplaybreaks
	
	\section{Introduction and motivation}\label{INTRO}
	\subsection{Furstenberg's intersection problem}
	Furstenberg (\cite{Fu2}) posed a series of fundamental questions on the transversality of dynamical systems. A central heuristic of those problems is that $\times 2, \times 3$ should be two very different actions. Towards this direction, we have the following conjecture. 
	\begin{conj}[The strong Furstenberg's intersection problem]\label{Conj}
		Let $A_2,A_3$ be closed $\times 2,\times 3$ invariant sets respectively and such that $\Haus A_2+\Haus A_3<1$. Then the intersection $A_2\cap A_3$ contains only rational numbers.
	\end{conj}
	Here $\Haus$ is the Hausdorff dimension, and it will be discussed in Section \ref{DIM}. The original form of the above conjecture states that for any irrational number $a,$ its $\times 2\mod 1$ orbit closure and $\times 3\mod 1$ orbit closure have total Hausdorff dimension at least $1$. In this direction, an earlier result of Furstenberg in \cite{Fu1} states that the $\times 2, \times 3\mod 1$ orbit, i.e.
	\[
	\{2^k3^m a\mod 1\}_{k,m\geq 1}
	\]
	is dense in $[0,1].$ As we mentioned above, Furstenberg posed a series of deep conjectures of this kind. Conjecture \ref{Conj} is perhaps the strongest among all of them. Some weaker conjectures are recently resolved, and we now introduce them.  In \cite{Fu2}, Furstenberg introduced the notion of CP-chains and showed that if $\Haus A_2+\Haus A_3<0.5$, then $\Haus A_2\cap A_3=0.$ One of the first breakthroughs of Conjecture \ref{Conj} is the following result which appeared in \cite{HS12}. This result settled an earlier sumset conjecture of Furstenberg.
	\begin{thm}[Hochman-Shmerkin]
		Let $A_2,A_3$ be closed $\times 2,\times 3$ invariant sets respectively. For all real numbers $u,v$ such that $uv\neq 0$ we have the following result
		\[
		\Haus(uA_2+vA_3)=\max\{1,\Haus A_2+\Haus A_3\}.
		\]
	\end{thm}
	Intuitively speaking, the above result says that the projections of $A_2\times A_3$ along non-trivial directions (i.e. not horizontal nor vertical) all attain the largest possible dimension. This somehow indicates that the fibres (i.e. intersections with lines orthogonal with the projecting direction) should not be large. A precise form of the intersection result was proved in \cite{Sh} and \cite{Wu} independently which improves Furstenberg's original result. The following result settled an earlier intersection conjecture of Furstenberg. 
	\begin{thm}[Shmerkin, Wu]\label{SW}
		Let $A_2,A_3$ be closed $\times 2,\times 3$ invariant sets respectively. For all real numbers $u,v$ such that $u\neq 0$ we have the following result
		\[
		\ubox(A_2\cap (uA_3+v))\leq \min\{0,\Haus A_2+\Haus A_3-1\}.
		\]
	\end{thm}
	Here $\ubox$ is the upper box dimension and it will be discussed in Section \ref{DIM}. The above theorem is a significant step towards Conjecture \ref{Conj}. Notice that when $\Haus{A_2}+\Haus{A_3}<1$ then we have $\ubox A_2\cap A_3=0.$ This indicates that $A_2\cap A_3$ is small which strongly supports Conjecture \ref{Conj}. From here, one might be curious about whether anything can be obtained beyond dimension zero. Indeed, this is one of the main focus of this paper. There are other recent results around Furstenberg's intersection problem. See \cite{Al20}, \cite{Au20}, \cite{BHY18}, \cite{Y20} for further details.
	
	We will show a stronger version of Furstenberg's intersection result. In the statement, we encounter the notions of the Hausdorff dimension ($\Haus$), the Assouad dimension ($\Assouad$), densities and sparseness and invariant sets. They are defined and discussed in details in Sections \ref{DIM}, \ref{DEN}, \ref{SPAR} and Section \ref{INV}. For concreteness, the results we list here are about $\times 2,\times 3\mod 1$ invariant sets. They still hold if we replace $2,3$ by $p,q$ respectively such that $\log p/\log q\notin\mathbb{Q}.$ The bound $O(N^{27s})$ below needs to be changed to $O(N^{C(p,q)s})$ with constants $C(p,q)$ depending on $p,q.$ 
	
	We shall see that being sparse implies having zero dimension. The advantage of considering sparseness is that it provides us with more quantitative understanding of the size of sets. For example, instead of considering an individual intersection, one can consider all intersections at once and the box-counting numbers can be controlled in a uniform way. We first introduce the following notion of spareness which provides us with a quantitative view of the 'topological smallness' of a compact subset of a line. 
	\begin{defn}\label{SPA}
		Let $A\subset\mathbb{R}$ be a compact set. We see that $A$ is (super) sparse near $0$ if 
		\[
		W(A)=\{k\in\mathbb{N}: A\cap ([2^{-k-1},2^{-k}]\cup [-2^{-k},-2^{-k-1}])\neq \emptyset\}
		\]
		has upper (Banach) density $0$. We call $W(A)$ the sparse index of $A$ near $0$. More generally given any $a\in A$, we define the sparse index of $A$ near $a$ by $W(A,a)=W(A-a)$ and we say that $A$ is (super) sparse near $a$ if and only if $A-a$ is (super) sparse near $0$. In general when $l\subset\mathbb{R}^2$ is contained in a line not parallel with the $Y$-coordinate axis, we define $W(A,a)$ as
		\[
		W(A,a)=W(\pi_Y(A),\pi_Y(a)).
		\]
	\end{defn} 
	Now we state the main results of this paper.
	\begin{thm}\label{THMHALF}
		Let $A_2,A_3$ be closed $\times 2, \times 3$ invariant sets respectively and $\Haus A_2+\Haus A_3=s<1/2$. Then let $l=l_{u,v}=A_2\cap (u A_3+v), u\neq 0$ be an intersection. The distance set $|l-l|$ is supersparse near $0$. Moreover, we have the following bound which is uniform with respect to $k\in\mathbb{N},$
		\[
		\#(W(|l-l|)\cap [k+1,k+N])=O(N^{27s}).
		\]
	\end{thm}
	\begin{thm}\label{THMONE}
		Let $A_2,A_3$ be closed $\times 2, \times 3$ invariant sets respectively and $\Haus A_2+\Haus A_3<1$. Then let $l_{u,v}=A_2\cap (u A_3+v), u\neq 0$ be an intersection. For all $a\in l_{u,v}$, $l_{u,v}$ is super sparse near $a$. Moreover, for each $\gamma>0,$ the following bound is uniform with respect to $|u|\in (\gamma,\gamma^{-1})$, $a\in l_{u,v}, k\in\mathbb{N},$
		\[
		\#(W(l_{u,v},a)\cap [k+1,k+N])=o(N).
		\]
	\end{thm}
	See Section \ref{DEN} for the meaning of densities. A direct consequence of the above theorems is the following corollary on a uniform version of box dimension estimate. This follows by applying Theorem \ref{THMONE}, Proposition \ref{USPA} and the discussions below Proposition \ref{USPA}. We remark that this uniform box dimension estimate is very useful in some problems concerning numbers with restricted digits, see \cite{BHY18}. 
	\begin{cor}[A uniform estimate for the upper box dimension]\label{UniBox}
		Let $A_2,A_3$ be closed $\times 2, \times 3$ invariant sets respectively and $\Haus A_2+\Haus A_3<1$. For each pair of real numbers $u,v$ with $u\neq 0,$ let $l_{u,v}=A_2\cap (u A_3+v), u\neq 0$ be the intersection. Let $\gamma>0$ be fixed. For each $\epsilon>0$ there is an integer $N_\epsilon$ such that whenever $|u|\in (\gamma,\gamma^{-1}), v\in\mathbb{R}$
		\[
		N(l_{u,v},2^{-N})\leq 2^{\epsilon N}
		\]  
		for all $N\geq N_\epsilon.$
	\end{cor}
	A key tool for proving the above results is the notion of sparseness. This is where we make use of the notation $W(.)$( see Section \ref{SPAR}). Our notion of sparseness is a very natural indicator which says that a set is small. We think this topic might be interesting on its own and we will give a detailed treatment of the notion of sparseness. In Section \ref{SPAR} we will prove some relations between sparseness and fractal dimensions. In particular, we show that being sparse implies having zero dimension but the converse is in general not true. In particular, we have the following results as consequences:
	\begin{itemize}
		\item By Proposition \ref{LSPA}: If $\Haus A_2+\Haus A_3=s<1/27$ then $\mathcal{H}^g(l_{u,v})=0$ for the gauge function $g(x)=\exp(-(-\log x)^{27s}).$
		\item By Proposition \ref{USPA}: If $\Haus A_2+\Haus A_3<1$ then $\Assouad l_{u,v}=0.$
	\end{itemize}
	In other words, $l_{u,v}$ is actually far away from having positive dimension if $\Haus A_2+\Haus A_3$ is small. The above weaker consequences already revisit partially the results (see Theorem \ref{SW} above) of \cite{Sh} and \cite{Wu} on Furstenberg's intersection conjecture. We note here that one can modify Wu's and Shmerkin's methods to show that $\Assouad l_{u,v}=0$ as well. 
	\subsection{Wu's ergodic sampling theorem}
	A key step for Wu's approach of Furstenberg's intersection problem is an ergodic sampling theorem proved as \cite[Theorem 6.1]{Wu}. The precise statement is rather technical and we will give a detailed discussion in Section \ref{SINAI}. Here we only provide some heuristics. Let $S=\{x_n\}_{n\geq 1}$ be a sequence of numbers in $[0,1]$, and we want to sample this sequence by taking a subsequence $S'=\{x_{n_{k}}\}_{k\geq 1}.$ Suppose that $S$ equidistributes in $[0,1],$ and taking a random subsequence $S'$ in a Bernoulli manner (by tossing a coin for each $x_n$ and decide whether or not put it into $S'$) will almost surely keep the equidistribution property. Now, instead of randomly (in a Bernoulli manner) choosing the subsequence $S'$, we can choose it according to some returning time with respect to an auxiliary dynamical system $(Y,T)$ given $y\in Y$ and $A\subset Y:$
	\[
	n_k=\text{the integer $i$ such that $T^{i}(y)\in A$ for the $k$-th time}.
	\]
	If $(Y,T)$ is a `complicated' system, then we expect that some randomness can be extract and the subsequence $S'$ will be still nicely distributed. \cite[Theorem 6.1]{Wu} provides a precise result when $S$ is an irrational rotation orbit and $(Y,T)$ is a dynamical system with positive entropy. We will provide a generalization of this result, see Theorem \ref{Cor}. Essentially, our result allows us to take $S$ to be a subsequence rather than the whole irrational rotation orbit.  Although not straightforwardly, we remark that one can actually modify Wu's proof in \cite{Wu} for proving Theorem \ref{Cor}. We decide to include a full detailed proof for this theorem in Section \ref{SINAI}. There are other results in this direction, see \cite{Au20} and \cite{Yu19}.
	
	\section{structure of this paper}
	In Section \ref{Notation}, we briefly recall some basic terminology from dynamical systems, notions of dimensions and densities of integer sequences. In Section \ref{Pre}, we introduce notions of sparseness and their connections with fractal dimensions. We point out the importance of Section \ref{Dip}, the dipole direction structure will be useful later. In Section \ref{IRR}, we prove some target hitting estimates using discrepancy theory and use it in Section \ref{SMALL} for proving Theorem \ref{THMHALF}. We present in Section \ref{SINAI} a version of Sinai's factor theorem which is closely related to but different than the version which appeared in \cite[Section 6]{Wu}. Then finally in Section \ref{LARGE} we proof Theorem \ref{THMONE}.
	\section{notation}\label{Notation}
	\subsection{Filtrations, atoms and entropy}
	Let $X$ be a set with $\sigma$-algebra $\mathcal{X}.$ A filtration of $\sigma$-algebras is a sequence $\mathcal{F}_n\subset\mathcal{X},n\geq 1$ such that
	\[
	\mathcal{F}_1\subset \mathcal{F}_2\subset\dots \subset \mathcal{X}.
	\]
	Given a measurable map $S:X\to X$ and a finite measurable partition $\mathcal{A}$ of $X$, we denote $S^{-n} \mathcal{A}$ to be the following finite collection of sets (notice that $S$ might not be invertible)
	\[
	\{S^{-n}(A) :A\in \mathcal{A}\}.
	\]
	Then we use $\vee_{i=0}^{n-1} S^{-i}\mathcal{A}$ to denote the $\sigma$-algebra generated by $S^{-i} \mathcal{A},i\in [0,n-1].$ An atom in $\vee_{i=0}^{n-1} S^{-i}\mathcal{A}$ is a set $A$ that can be written as
	\[
	A=\bigcap_{i} C_i
	\]
	where for each $i\in\{0,\dots,n-1\}$,  $C_i\in S^{-i}\mathcal{A}.$ In other words, an atom in $\vee_{i=0}^{n-1} S^{-i}\mathcal{A}$ can be also described as follows. Given a sequence $\{A_i\}_{i=0}^n\in \mathcal{A}^{n+1}$, we define the following set (which can be empty)
	\[
	\{x\in X: x\in A_0, S(x)\in A_1, \dots, S^n(x)\in A_n\}.
	\]  
	The above set is an atom and all atoms have the above form. In this sense $\vee_{i=0}^{n-1} S^{-i}\mathcal{A}$ is generated by a finite partition $\mathcal{A}_{n-1}$ of $X$ which is finer than $\mathcal{A}.$ Let $\mu$ be a probability measure. Then we define the entropy of $\mu$ with respect to a finite partition $\mathcal{A}$ as follows
	\[
	H(\mu,\mathcal{A})=-\sum_{A\in\mathcal{A}} \mu(A)\log \mu(A).
	\]
	We define the entropy of $S$ as follows
	\[
	h(S,\mu)=\lim_{n\to\infty} \frac{1}{n} H(\mu, \mathcal{A}_{n-1}),
	\]
	where $\mathcal{A}$ is a partition such that $\vee_{i=1}^\infty S^{-i}\mathcal{A}=\mathcal{X}.$ Here we implicitly used Sinai's entropy theorem, see \cite[Lemma 8.8]{PY}.
	\subsection{Dynamical systems and factors}
	A measurable dynamical system is denoted as $(X,\mathcal{X},S,\mu)$ where $X$ is a set with $\sigma$-algebra $\mathcal{X}$ and measure $\mu$ and a measurable map $S:X\to X.$ In case when $\mathcal{X}$ is clear in context (for example Borel $\sigma$-algebra in Borel spaces) then we do not explicitly write it down. Given two dynamical systems $(X,\mathcal{X},S,\mu),(X_1,\mathcal{X}_1,S_1,\mu_1)$, a measurable map $f:X\to X_1$ is called a factorization map and $(X_1,\mathcal{X}_1,S_1,\mu_1)$ is called a factor of $(X,\mathcal{X},S,\mu)$ if $\mu_1=f\mu$ and $f\circ S=S_1\circ f.$
	\subsection{Dynamics on product sets and components}
	Let $(X,S,\mu)$ be a measurable dynamical system with $X=X_1\times X_2.$ Denote the projection function $\pi_1: X\to X_1.$ Then the $X_1$ component of the measure $\mu$ is the projected measure $\pi_1 \mu.$ Let $\mathcal{A}$ be a collection of subsets of $X$. The $X_1$ component of $\mathcal{A}$ is $\pi_1 \mathcal{A}.$ In the case when $S$ is a product or skew-product of maps, namely, for $(x_1,x_2)\in X$, $S(x_1,x_2)=(S_1(x_1),S_2(x_1,x_2)),$ then $(X_1,S_1,\pi_1 \mu)$ is a factor of $(X,S,\mu)$ and $\pi_1 \mu$ is $S_1$-invariant if $\mu$ is $S$-invariant. We call $(X_1,S_1,\pi_1 \mu)$ the $X_1$ component of $(X,S,\mu).$ 
	
	\subsection{Equidistribution}
	Let $X=\{x_n\}_{n\geq 1}$ be a sequence in $[0,1].$ We say that $X$ equidistributes in $[0,1]$ if for each interval $[a,b]\subset [0,1]$ we have the following result,
	\[
	\lim_{N\to\infty} \frac{1}{N} \sum_{n=1}^N \mathbbm{1}_{[a,b]}(x_n)=b-a=\lambda([a,b]),
	\]
	where $\lambda$ is the Lebesgue measure on $[0,1].$
	\subsection{Dimensions}\label{DIM}
	We will encounter (and have encountered) in this paper various notions of fractal dimensions. We briefly introduce the definitions.  For more details on the Hausdorff and box dimensions, see \cite[Chapters 2,3]{Fa} and \cite[Chapters 4,5]{Ma1}. For the Assouad dimension, see \cite{F}. We shall use $N(F,r)$ for the minimal covering number of a set $F$ in $\mathbb{R}^n$ with cubes of side length $r>0$. 
	
	\subsubsection{Hausdorff dimension}
	
	Let $g: [0,1)\to [0,\infty)$ be a continuous function such that $g(0)=0$. Then for all $\delta>0$ we define the following quantity
	\[
	\mathcal{H}^g_\delta(F)=\inf\left\{\sum_{i=1}^{\infty}g(\mathrm{diam} (U_i)): \bigcup_i U_i\supset F, \mathrm{diam}(U_i)<\delta\right\}.
	\]
	The $g$-Hausdorff measure of $F$ is
	\[
	\mathcal{H}^g(F)=\lim_{\delta\to 0} \mathcal{H}^g_{\delta}(F).
	\]
	When $g(x)=x^s$ then $\mathcal{H}^g=\mathcal{H}^s$ is the $s$-Hausdorff measure and Hausdorff dimension of $F$ is
	\[
	\Haus F=\inf\{s\geq 0:\mathcal{H}^s(F)=0\}=\sup\{s\geq 0: \mathcal{H}^s(F)=\infty          \}.
	\]
	\subsubsection{Box dimensions}
	The upper box dimension of a bounded set $F$ is
	\[
	\overline{\boxd} F=\limsup_{r\to 0} \left(-\frac{\log N(F,r)}{\log r}\right).
	\]
	Similarly the lower box dimension of $F$ is
	\[
	\lbox F=\liminf_{r\to 0} \left(-\frac{\log N(F,r)}{\log r}\right).
	\]
	If the limsup and liminf are equal we call this value the box dimension of $F$ and we denote it as $\boxd F.$
	\subsubsection{Assouad and modified Assouad dimensions}
	The Assouad dimension of $F$ is 
	\begin{align*}
	\Assouad F = \inf \Bigg\{ s \ge 0 \, \, \colon \, (\exists \, C >0)\, (\forall & R>0)\,  (\forall r \in (0,R))\, (\forall x \in F) \\ 
	&N(B(x,R) \cap F,r) \le C \left( \frac{R}{r}\right)^s \Bigg\},
	\end{align*}
	where $B(x,R)$ denotes the closed ball of centre $x$ and radius $R$.
	
	The modified Assouad dimension of $F$ is
	\[
	\pAssouad F=\inf\left\{\sup_{i\in\mathbb{N}}\{\Assouad F_i\}: F\subset\cup_i F_i\right\}.
	\]
	In particularly any countable set has modified Assouad dimension $0$ and it is easy to see that
	\[
	\pAssouad F\leq \Assouad F.
	\] 
	
	\subsection{$\times p\mod 1$ invariant sets}\label{INV}
	In this paper, given an integer $p\geq 2$, we use $A_p$ to denote an arbitrary closed $\times p\mod 1$ invariant subset of $[0,1]$. This is to say, for all $a\in A_p$, $\{pa\}\in A_p$, where $\{x\}$ is the fractional part of $x$. We say that $A_p$ is strictly invariant if $a\in A_p\iff \{pa\}\in A_p.$ For each closed $\times p\mod 1$ invariant set $A_p$, it is known (\cite[Theorem 5.1]{Fu}) that $\Haus A_p=\ubox A_p$. In particular for any integers $p,q\geq 2$, $\Haus A_p\times A_q=\ubox A_p\times A_q.$
	
	\subsection{Densities of integer sequences}\label{DEN}
	We also work with various notions of densities of integer sequences. We recall two notions of density for integer sequences.
	\begin{defn}
		The upper natural density of $W$ is defined as
		\[
		\overline{d}(W)=\limsup_{n\to\infty} \frac{\#\{W\cap [1,n]\}}{n}.
		\]
		Similarly, we define the lower natural density by replacing the above $\limsup$ with $\liminf$ and write it as $\underline{d}(W)$. If these two numbers coincide we call it the natural density of $W$ and write it as $d(W).$
	\end{defn}
	
	\begin{defn}
		The upper Banach density of $W$ is defined as
		\[
		d_B(W)=\limsup_{k,M\to\infty} \frac{1}{k}(\#\{W\cap [M,M+k-1]\}).
		\]
	\end{defn}
	\subsection{The big $O$ and small $o$ notations}
	Let $f,g: \mathbb{N}\to [0,\infty)$ be two functions. We write
	\[
	f=O(g)
	\]
	if there exists positive number $C>0$ such that
	\[
	f(k)\leq Cg(k)
	\]
	for all $k\in\mathbb{N}.$
	Similarly we write
	\[
	f=o(g)
	\]
	if for any $\epsilon>0$ there exists $N\in\mathbb{N}$ such that for all $k\geq N$ we have
	\[
	f(k)\leq \epsilon g(k).
	\]
	In some occasions there is another parameter set $S$ and we have functions $f,g:\mathbb{N}\times S\to [0,\infty).$ For each $c\in S$  we write $f=O_c(g), o_c(g)$ to indicate that the above tendencies depend on the choice of $c.$ We say that $f=O(g), o(g)$ uniformly for $c\in S$ if the above tendencies do not depend on the choice of $c.$
	\subsection{Weak convergence of measures and the Portmanteau theorem}
	In Section \ref{LARGE}, we need the notion of weak * convergence of measures and the Portmanteau theorem. Let $\mu_{k},k\geq 1$ be a sequence of probability measures on a Borel space $X.$ We say that $\mu_k\to \mu$ in weak * sense (or weakly) if for all bounded continuous functions $f:X\to\mathbb{R}$ 
	\[
	\lim_{k\to\infty} \int_X fd\mu_{k}=\int_X fd\mu.
	\]
	The following version of the Portmanteau theorem is taken from \cite[Theorem 13.16]{K06} and \cite[Theorem 1.3]{Su14}.
	\begin{thm}[Portmanteau theorem]\label{PORT}
		Let $\mu_k,k\geq 1$ be a sequence in $\mathcal{P}(X)$ (the space of Borel probability measures supported on $X$) where $X$ is a Borel space. Let $\mu\in \mathcal{P}(X).$ The following statements are equivalent:
		\begin{itemize}
			\item[1]: $\mu_k\to\mu$ weakly;
			\item[2]: $\limsup_{k} \mu_k(K)\leq \mu(K)$ for all closed subsets $K$ of $X$;
			\item[3]: $\lim_{k\to\infty} \int_X fd\mu_k=\int_X fd\mu$ for bounded and $\mu$-almost everywhere continuous real valued functions $f$ on $X.$
		\end{itemize}
	\end{thm}
	There are a lot of other equivalent statements for the Portmanteau theorem, for more details, see \cite{Su14} and the references therein. One particular use of the above result is related to invariant measures of almost continuous dynamical systems. More precisely, let $X$ be a compact metric space. Let $T:X\to X$ be a map (not necessary continuous). For each integer $n\geq 1,$ let $x_n\in X$ be arbitrarily chosen and let $\mu_n=(n+1)^{-1}\sum_{i=0}^{n} \delta_{T^i(x_n)}$ be a sequence of probability measures on $X.$ Let $\mu$ be a weak * limit point of this sequence. In the case when $T$ is continuous, we know that $\mu$ is $T$-invariant. This is the content of Kryloff-Bogoliouboff theorem. We can extend this result if $T$ is only assumed to be $\mu$-almost everywhere continuous. In fact, for any $f\in C(X),$ we have the following result for a sequence of integers $\{i_k\}_{k\geq 1}$ such that $i_k\to\infty$,
	\[
	\int_{X} f(x) d\mu(x)=\lim_{k\to\infty} \int_X f(x)d\mu_{i_k}(x).
	\] 
	Now we want to consider the same for the function $f\circ T.$ It is continuous at where $T$ is continuous. Then we see that $f\circ T$ is $\mu$-almost everywhere continuous. We have the following result (by Theorem \ref{PORT}(3)),
	\[
	\int_{X} f(T(x)) d\mu(x)=\lim_{k\to\infty} \int_X f(T(x))d\mu_{i_k}(x)=\lim_{k\to\infty}\frac{1}{i_k+1}\sum_{i=0}^{i_k} f(T(T^{i}(x_{i_k}))).
	\]
	The last term is equal to
	\[
	\lim_{k\to\infty}\frac{1}{i_k+1}\left(\sum_{i=0}^{i_k} f(T^{i}(x_{i_k}))+f(T^{i_k+1}(x_{i_k}))-f(x_{i_k})\right)
	\]
	which is the same as (recall that $f$ is bounded)
	\[
	\lim_{k\to\infty}\frac{1}{i_k+1}\sum_{i=0}^{i_k} f((T^{i}(x_{i_k}))=\int_{X} f(x) d\mu(x).
	\]
	This shows that $\mu$ is $T$-invariant. In general, it is not simple to show that $T$ is $\mu$-almost everywhere continuous. There are some special cases when this is possible to be checked, see for example \cite[Section 5.2]{Wu}.
	\section{Sparseness and Dipoles}\label{Pre}
	In this section, we introduce our key tools for approaching Furstenberg's intersection problem. The ideas behind the definitions are very natural and straightforward, however, we are not able to find any direct references. We decide to give a detailed treatment which will be more than what we need for Furstenberg's intersection problem.
	\subsection{Doubling measures}
	Later we shall use some facts about doubling measures. Here we are interested in doubling measures supported on compact subsets of $\mathbb{R}.$ We have the following result. The proofs can be found in \cite{VK}, \cite[Theorem 6.10]{Lu} and \cite{KRS}.
	\begin{thm}\label{DOUB}
		Let $A\subset\mathbb{R}$ be a compact set. Then there is a doubling probability  measure supported on $A$. Namely, there is a measure $\mu\in\mathcal{P}(A)$ and there exists an absolute constant (called the doubling constant for $\mathbb{R}$) $D\geq 1$ such that for all $a\in A$ and $r>0$,
		\[
		0<\mu(B(x,2r))\leq D\mu(B(x,r))< \infty.
		\] 
	\end{thm}
	According to \cite[Section 6.13]{Lu}, the constant $D$ for $\mathbb{R}$ can be chosen to be $2\times 3\times 4\times 9^5.$
	\subsection{Sparseness}\label{SPAR}
	We call the notion of spareness (Definition \ref{SPA}).
	\begin{defn*}
		Let $A\subset\mathbb{R}$ be a compact set. We see that $A$ is (super) sparse near $0$ if 
		\[
		W(A)=\{k\in\mathbb{N}: A\cap ([2^{-k-1},2^{-k}]\cup [-2^{-k},-2^{-k-1}])\neq \emptyset\}
		\]
		has upper (Banach) density $0$. We call $W(A)$ the sparse index of $A$ near $0$. More generally given any $a\in A$, we define the sparse index of $A$ near $a$ by $W(A,a)=W(A-a)$ and we say that $A$ is (super) sparse near $a$ if and only if $A-a$ is (super) sparse near $0$. In general when $l\subset\mathbb{R}^2$ is contained in a line not parallel with the $Y$-coordinate axis, we define $W(A,a)$ as
		\[
		W(A,a)=W(\pi_Y(A),\pi_Y(a)).
		\]
	\end{defn*} 
	It is easy to see that (super) sparseness is insensitive with respect to scaling. That is to say, if $W(A,a)$ has upper natural density $0$, then for each real number $c\neq 0$, $W(cA,ca)$ also has natural density $0$. A similar result holds for upper Banach density as well.
	
	Given any set $A\subset\mathbb{R}$, we denote $|A-A|$ as its distance set. Intuitively, if $|A-A|$ is sparse near $0$ then $A$ cannot be too large. A less restrictive notion is uniform sparseness. That is to say, for each $\delta>0,$ there is an integer $N_\delta$ such that $\#W(A,a)\cap [1,\dots,N]\leq \delta N$ for all $a\in A, N\geq N_\delta.$ Similar notion of uniform super sparseness can be formulated as well. In particular, if $|A-A|$ is (super) sparse then $A$ is uniformly sparse.
	\begin{prop}\label{USPA}
		We have the following results.
		\begin{itemize}
			\item{1:} Given any uniformly sparse set $A\subset\mathbb{R},$ we have $\ubox A=0.$ 
			
			\item{2:} If $A$ is uniformly super sparse then $\Assouad A=0.$  
			
			\item{3:} Let $A\subset [0,1]$ be a compact set such that for all $a\in A$, $A$ is super sparse near $a$ then 
			\[
			\pAssouad A=0.
			\]
			\item{4:} The converses of the above are in general not true. On the other hand, if $A$ is finite then it is uniformly super sparse.
		\end{itemize}
	\end{prop}
	\begin{proof}
		The proofs for (1) and (2),(3) are very similar and we only write the proof for (1). First observe that the last conclusion of (4) is trivial. We now illustrate the first part of (4). Let $A_0$ be the set $\{0\}\cup\{2^{-k}\}_{k\geq 0}\subset [0,1]$. We see that $\ubox A_0=0$ but we can see that $A_0$ is not sparse near $0$ and therefore it is not uniformly sparse. Now we consider (1). Let $A\in [0,1]$ be a uniformly sparse set. Then we see that the following set has $0$ upper natural density uniformly across $a\in A$,
		\[
		W(A,a)=\{k\in\mathbb{N}: |A-A|\cap [2^{-k-1},2^{-k}]\neq \emptyset\}.
		\]
		Since $A$ is compact we assume that it is contained in $[0,1]$. To bound the upper box dimension of $A$ we shall use Theorem \ref{DOUB} and find a doubling (with doubling constant $D>0$) probability measure supported on $A$.  Let $a\in A$ be arbitrarily chosen and for any integer $n\geq 0$ we can find a nested sequence of intervals $a\in B(a,2^{-n})\subset\dots\subset B(a,1).$ Since we assumed that $A\subset [0,1]$ therefore we see that $\mu(B(a,1))=1.$ Now we make use of the uniform sparseness of $|A-A|$. It is clear that if $A\cap\overline{B(a,2^{-j})\setminus B(a,2^{-j-1})}\neq\emptyset$ then $j\in W(A,a)$. This means that $A\cap B(a,2^{-j})=A\cap B(a,2^{-j-1})$ if $j\notin W(A,a).$ Then we write
		\[
		\mu(B(a,2^{-n}))=\mu(B(a,1))\prod_{j=0}^{n-1} \frac{\mu(B(a,2^{-j-1}))}{\mu(B(a,2^{-j}))}.
		\]
		If $j\notin W(A,a)$ then $A\cap\overline{B(a,2^{-j})\setminus B(a,2^{-j-1})}=\emptyset$ therefore we see that,
		\[
		\frac{\mu(B(a,2^{-j-1}))}{\mu(B(a,2^{-j}))}=1,
		\]
		otherwise if $j\in W(A,a)$ we can still write
		\[
		\frac{\mu(B(a,2^{-j-1}))}{\mu(B(a,2^{-j}))}\geq D^{-1}.
		\]
		Since $W(A,a)$ has natural density $0$ uniformly across $a\in A,$ we see that for all $\epsilon>0$ there exist a $N_\epsilon$ such that for all $a\in A, N\geq N_\epsilon$ we have
		\[
		\#W(A,a)\cap [1,N]\leq \epsilon N.
		\]
		Then we see that for all $N\geq N_\epsilon$
		\[
		\mu(B(a,2^{-N}))\geq D^{-\epsilon N}.
		\]
		We can cover $A$ with disjoint intervals of length $2^{-N-1}$ and denote the collection of such intervals as $\mathcal{N}_{N+1}$, then for any $I\in \mathcal{N}_{N+1}$ there is a $a\in I\cap A$ such that $I\subset B(a,2^{-N})$ and therefore $\mu(I)\geq D^{-\epsilon (N)}.$ Since $\mu$ is a probability measure we see that
		\[
		\#\mathcal{N}_{N+1}\leq D^{\epsilon N}.
		\]
		This implies that $\ubox A\leq \epsilon \log D/\log 2.$ and because $\epsilon$ can be arbitrarily chosen we see that $\ubox A=0.$ This finishes the proof of (1). The proofs of (2), (3) are similar and we omit the details.
	\end{proof}
	Since the doubling constant $D$ can be chosen independently with respect to $A$ we see that with the uniform sparseness assumption, for each $\epsilon>0$ there is an integer $N_\epsilon$ such that
	\[
	N(A,2^{-n})\leq D^{\epsilon n}
	\]
	for all $n\geq N_\epsilon.$ Therefore, if we have a collection of compact sets $\{A_i\}_{i\in\mathcal{I}}$ of $[0,1]$ and we assume uniform sparseness uniformly across $i\in\mathcal{I}$ then for each $\epsilon>0$ there is an integer $N_\epsilon$ and for all $n\geq N_\epsilon, i\in\mathcal{I}$ we have
	\[
	N(A_i,2^{-n})\leq D^{\epsilon n}.
	\]

	In general, if we can control $W(A,a)$ individually for all $a\in A$ then it is possible to say something about the Hausdorff measure of $A$ with respect to a certain gauge function. 
	
	\begin{prop}\label{LSPA}
		Let $A\subset [0,1]$ be a compact set. Let $f:[0,\infty)\to [0,\infty)$ be such that for all $a\in A$,
		\[
		\#W(A,a)\cap [1,N]=o_a(f(N)) \text{ as $N\to \infty.$}
		\]
		We write a gauge function as $g(x)=\exp(-f(1-\log x/\log 2))$ for $x\in (0,1).$ Then we have $\mathcal{H}^{g}(A)<\infty.$ 
	\end{prop} 
	
	\begin{proof}
		Since $A$ is compact we can find a doubling probability measure with doubling constant $D$ on it, see Theorem \ref{DOUB}. Let $c>0$ be an arbitrarily chosen constant. Then for each $a\in A$, because of the sparseness of $A$ around $a$ with a similar argument as in the proof of Lemma \ref{USPA} we see that there exists an integer $N_{a}$ such that whenever $N\geq N_{a}$ we have
		\[
		\mu(B(a,2^{-N}))\geq D^{- f(N)/c}.
		\]
		Since $A$ is compact we see that there is a finite collection of intervals of form $I_a=B(a,N_a)$ that covers $A$. By Besicovitch's covering lemma(\cite[Chapter 2, Section 7]{Ma1}) we see that there exists an absolute constant $C>0$ and for a finite subset $A'$ of $A$ such that
		\[
		\sum_{a\in A'} D^{-f(N_a)/c}\leq \sum_{a\in A'} \mu(I_a)\leq C.
		\]
		Now we choose $c=\max\{\log D,1\}.$ Notice that $f(N_a)=f(-\log 2^{-N_a}/\log 2)$ and $2^{-N_a+1}$ is the length of $D_{k_a}(a)$.  Let $r_a$ be the length of $I_a$, we see that for an absolute constant $C>0$,
		\[
		\sum_{a\in A'}\exp( - f(1-\log r_a/\log 2))\leq C.
		\]
		It is clear that we can bound $\max_{a\in A'} r_a$ to be arbitrarily small. This implies that $\mathcal{H}^g(A)<\infty.$
	\end{proof}
	In particular if $f(N)=o(N)$ then we see that $\Haus A=0$. If $f(N)=o(N^{\sigma})$ for $\sigma\in (0,1)$ then we can choose $g(x)=\exp(-(-\log x)^\sigma).$

	\subsection{Dipole set}\label{Dip}
	Let $C>1.$ Then there is a constant $c>0$ such that the following holds for all small enough $\delta>0:$
	
	Let $A\subset\mathbb{R}^2$ be a compact subset. Let $E\subset [0,2\pi]$ be a $\delta$-separated set of directions. Suppose that for each $e\in E$ we can find $x_e,y_e\in A$ such that
	\[
	|y_e-x_e|\in [C^{-1},C]
	\]
	and $y_e-x_e$ points towards the direction $e$. Then we have $N(A,\delta)\geq c\sqrt{\#E}.$ 
	
	To see this, we only need to cover $A$ with disjoint $\delta$-cubes.  As long as $\delta$ is small enough, there is a number $c'>0$ so that if there is a $\delta$-cube containing $M$ points of form $x_e$ then the corresponding $y_e$ are all at least $c'\delta$-separated from each other. Therefore we have $N(A,\delta)\geq c'M/2.$ On the other hand if non of the $\delta$-cubes contain more than $M$ many points of form $x_e$ then $N(A,\delta)\geq \#E/M.$ Then we see that for all integer $M$,
	\[
	N(A,\delta)\geq \max\left\{  c'M/2,            \#E/M        \right\}\geq \sqrt{c'/2}\sqrt{\#E}.
	\]
	\begin{defn}
		Let $A\subset\mathbb{R}^2$ be a compact subset, the dipole direction set of $A$ is defined as follows,
		\[
		DD(A)=\left\{\frac{x-y}{|x-y|}:|x-y|\in [1/6,1.5], x,y\in A\right\}.
		\]
	\end{defn}
	It is easy to see that when $A$ is compact $DD(A)$ is also compact. We have shown the following lemma.
	\begin{lma}
		For all compact subset $A\subset \mathbb{R}^2,$ we have the following result
		\[
		\ubox A\geq 0.5\ubox DD(A).
		\]
	\end{lma}
	\section{Irrational rotations, discrepancy and target hitting}\label{IRR}
	
	Let $\alpha\in (0,1)$ be an irrational number. Consider the rotation system $R_{\alpha}: [0,1]\to [0,1]$ defined as follows,
	\[
	R_\alpha(x)=x+\alpha \mod 1.
	\]
	Then for any compact subset $A\subset [0,1]$, by Birkhoff's ergodic theorem(\cite[Theorem 10.6]{PY}) we see that for Lebesgue almost all $x\in [0,1],$
	\[
	\lim_{N\to\infty}\frac{1}{N} \sum_{n=0}^{N-1}\mathbbm{1}_A(R^n_\alpha(x))=\lambda(A),
	\]
	where $\lambda$ is the Lebesgue measure on $[0,1].$ This implies that when $A$ is small we expect that $n\geq 0, R^n_\alpha(x)\in A$ happens not so often. It is known that the circle rotation system with irrational $\alpha$ is uniquely ergodic therefore it is expected that $R^n_\alpha(0)\in A$ happens not so often. 
	\begin{lma}\label{Target}
		Let $A\subset [0,1]$ be a compact set and for any irrational number $\alpha\in (0,1)$ we construct the following sequence,
		\[
		W=\{k\in\mathbb{N}: R^{k}_\alpha(0)=\{k\alpha\}\in A \}.
		\]
		Then the upper Banach density of $W$ is at most $\lambda(A).$
	\end{lma}
	\begin{proof}
		For any $\epsilon>0,$ we can cover $A$ with intervals $A\subset \bigcup_{i\in\mathcal{I}}I_i$ such that $\mathcal{I}$ is a finite set and \[\sum_{i\in\mathcal{I}} \lambda(I_i)\leq \lambda(A)+\epsilon.\] Then we can approximate each $\mathbbm{1}_{I_i}$ with a continuous function $f_i:[0,1]\to [0,1]$  such that $f_i(x)=1$ for $x\in I_i$ and $f_i(x)=0$ for $x\notin (1+\epsilon)I_i,$ where $(1+\epsilon)I_i$ is the interval with the same centre as $I_i$ but its length is equal to $(1+\epsilon)$ times that of $I_i.$ Then because of the unique ergodicity we see that for each $i\in\mathcal{I}$ and $x\in [0,1),$
		\[
		\lim_{N\to\infty}\frac{1}{N}\sum_{i=0}^{N-1}f_i(R^i_\alpha(x))=\int f_i d\lambda.
		\]
		Furthermore the above limit holds uniformly across $x\in [0,1).$ Therefore for each $i\in\mathcal{I}$ there is a number $N_i$ which does not depend on $x$ such that for each $N\geq N_i$ and $x\in [0,1),$
		\[
		\frac{1}{N}\sum_{i=0}^{N-1}f_i(R^i_\alpha(x))\leq (1+\epsilon)\int f_id\lambda\leq (1+\epsilon)^2\lambda(I_i).
		\]
		Now let $N_\epsilon=\max_{i\in\mathcal{I}} N_i$ (this is where we use the finiteness of $\mathcal{I}$) and we see that for any integers $a,M$ such that $M\geq N_\epsilon$ we see that
		\[
		W\cap [a+1,a+M]\subset \{k\in [a+1,a+M]: \{k\alpha\}\in A   \}\subset \bigcup_{i\in\mathcal{I}} \{k\in [a+1,a+M]: \{k\alpha\}\in I_i   \}.
		\]
		Since $M\geq N_\epsilon$ for each $i\in\mathcal{I}$ we see that
		\[
		\#\{k\in [a+1,a+M]: \{k\alpha\}\in I_i   \}\leq \sum_{k=0}^{M-1} f_i(R^i_\alpha(R^a_\alpha(0)))\leq (1+\epsilon)^2\lambda(I_i)M.
		\]
		This implies that
		\[
		\#W\cap [a+1,a+M]\leq \sum_{i\in\mathcal{I}}(1+\epsilon)^2\lambda(I_i)M\leq (1+\epsilon)^2(\lambda(A)+\epsilon)M. 
		\]
		Since $\epsilon>0$ and $M>N_{\epsilon}$ can be chosen arbitrarily we see that the upper Banach density of $W$ is at most $\lambda(A).$
		
	\end{proof}
	
	It is natural to consider what happens when $A$ is small in dimension. For this purpose we need to consider error terms in ergodic limits. We will be most interested in the cases when $\alpha=\log p/\log q$ and $\alpha\notin\mathbb{Q}.$ It was proved in \cite{B15} that there are numbers $C(\alpha),c(\alpha)>0$ such that for all integers $n,m\geq 1$
	\[
	\tag{1}\left| \alpha-\frac{m}{n}\right|\geq c(\alpha)\frac{1}{n^{C(\alpha)}}.
	\]
	The best known example in this kind is when $\alpha=\log 2/\log 3$ and in this case, see \cite[ proposition and formula (6)(7) on page 160]{R85} the above inequality can be written as
	\[
	\left| \alpha-\frac{m}{n}\right|\geq 0.00000000000001\frac{1}{n^{14.3}}.
	\]
	Since for any two different integers $i_1,i_2$ we have
	\[
	|\{i_1\alpha\}-\{i_2\alpha\}|=|i_1\alpha-i_2\alpha-M_1+M_2|,
	\]
	where $M_1=\lfloor i_1\alpha \rfloor,M_2=\lfloor i_2\alpha \rfloor$. We see that
	\[
	|i_1\alpha-i_2\alpha-M_1+M_2|=|i_1-i_2|\left|\alpha-\frac{M_1-M_2}{i_1-i_2}\right|.
	\]
	As we can assume that $i_1>i_2$, the inequality (1) implies that
	\[
	|\{i_1\alpha\}-\{i_2\alpha\}|=|i_1-i_2|\left|\alpha-\frac{M_1-M_2}{i_1-i_2}\right|\geq c(\alpha) \frac{1}{i_1^{C(\alpha)-1}}.\tag{GAP}
	\]
	This is the key point of the inequality (1) that we shall use.
	\begin{lma}\label{LM2}
		Let $A\subset [0,1]$ be a set with $\ubox A=s<1$, then for any irrational number of form $\alpha=\log p/\log q$ with two integers $p,q>0$, we have the following inequality holds for all $\epsilon\in (0,1-s),$
		\[
		\sum_{n=0}^{N-1}\mathbbm{1}_{A}(R^n_\alpha(0))=O_{\alpha,\epsilon}(N^{C(\alpha)(s+\epsilon)}),
		\]
		where $C(\alpha)>0$ is a constant depends only on $\alpha$.
	\end{lma}
	\begin{rem}
		This lemma applies better in the case when $s$ is small. For example if $\alpha=\log 2/\log 3$ then when $s<1/14$ we have the following polynomial bound,
		\[
		\sum_{n=0}^{N-1}\mathbbm{1}_{A}(R^n_\alpha(0))=O(N^{0.95}).
		\]
	\end{rem}
	\begin{proof}
		Let $N$ be a large integer and we consider the following sequence
		\[
		S_N(\alpha)=\{\{i\alpha\}\}_{i\in [0,N]}.
		\]
		Then it is clear that elements in $S_N(\alpha)$ never coincide because $\alpha$ is an irrational number. Then by inequality $(\mathrm{GAP})$ we see that there exist positive numbers $c(\alpha), C(\alpha)>0$ such that for any $x,y\in S_N(\alpha)$ with $x\neq y,$
		\[
		|x-y|\geq c(\alpha) N^{-C(\alpha)+1}.
		\]
		Now we choose $r_N=N^{-C(\alpha)+1}$ and cover $A$ with $k_N=N(A,r_N)$ many disjoint $r_N$-intervals. We denote again the union of those $r_N$-intervals as $A^{r_N}$. Then we see that 
		\[
		\sum_{n=0}^{N-1}\mathbbm{1}_{A^{r_N}}(R^n_\alpha(0))=O_{\alpha}(k_N).
		\]
		This is because each $r_N$-interval we use to cover $A$ contains at most $O_\alpha(1)$ many points in $S_N(\alpha).$ Then because of the dimension requirement of $A$ we see that for any $\epsilon>0$,
		\[
		k_N=O_{\epsilon}(r_N^{-s-\epsilon}).
		\]
		Therefore we see that for a constant $C'(\alpha)$ we have
		\[
		\sum_{n=0}^{N-1}\mathbbm{1}_{A^{r_N}}(R^n_\alpha(0))=O_{\alpha,\epsilon}(N^{C'(\alpha)(s+\epsilon)}).
		\]
		This proves the result.
	\end{proof}
	\section{Small sets, dipole configurations}\label{SMALL}
	In this section we study $A_2\cap (uA_3+v)$ when $\Haus A_2+\Haus A_3<1/2$. Furstenberg studied this intersection in \cite{Fu2}, in particular, he showed that $\Haus A_2\cap (uA_3+v)=0$ for all $u\neq 0.$ In this section, we provide a more straightforward argument and a stronger result.
	\begin{proof}[Proof of Theorem \ref{THMHALF}]
		We consider the product set $K=A_3\times A_2.$ Then $l$ is, up to rescaling, the same as $l_K=l'\cap K$ with $l'=\{y=ux+v\}.$ For convenience we require that $u>1$ but we note that the cases for other $u\neq 0$ are similar. Now we want to show that $|l_K-l_K|$ is super sparse near $0$. We denote $W_K=W(|l_K-l_K|)$ and we want to show that $W_K$ has zero upper Banach density. Now for each $k\in W_K$ we can find $x_k,y_k\in l_K$ such that
		\[
		|y_k- x_k|\in [2^{-k-1},2^{-k}].
		\]
		Without loss of generality we shall assume that the vector $y_k-x_k$ has positive $Y$-component. Now let $\alpha=\log 2/\log 3$ and we can construct the map $T=\mathbb{R}^2\times [0,1]\to \mathbb{R}^2\times [0,1],$
		\[
		T((t_1,t_2),t)=
		\begin{cases*}
		((t_1,2t_2),R_\alpha(t)), \text{ if $t+\alpha\leq 1$}\\
		((3t_1,2t_2),R_\alpha(t)), \text{ if $t+\alpha>1$} 
		\end{cases*}
		\]
		Now let $x,y\in\mathbb{R}^2$ be two different points such that the line segment $xy$ is not parallel to the coordinate axis. Then we can find the following sequence of pairs of points in $\mathbb{R}^2$,
		\[
		((x_n,t_n)=T^n(x,0),(y_n,t_n)=T^n(y,0))_{n\geq 0}.
		\]
		Now construct the following sequences,
		\[
		(\theta_{1,n},\theta_{2,n})=\frac{y_n-x_n}{|x_n-y_n|}\in S^1, \theta_n(x,y)=\log \left(\frac{\theta_{2,n}}{\theta_{1,n}}\right).
		\]
		Then we see that $\theta_0(x,y)=\log \left(\frac{\theta_{2,0}}{\theta_{1,0}}\right)$ and in general for each integer $n\geq 1,$
		\[
		\theta_n(x,y)=\log 3\{n\log 2/\log 3\}+\theta_0.
		\]
		Now we apply the above map $T$ for $k$ times with initial pair $x_k,y_k$ and end up with the pair
		\[
		((x,t_k)=T^{k}(x_k,0),(y,t_k)=T^{k}(y_k,0)).
		\]
		Then we see that $\theta_k(x_k,y_k)=\log 3\{k\log 2/\log 3\}+\theta_0(x_k,y_k)\in (\log u,\log u+\log 3).$ We want to estimate the distance $|x-y|$, the $Y$-component of $y_k-x_k$ lies in $[2^{-k-1}/\sqrt{2},2^{-k}].$ Therefore we see that the $Y$-component of $y-x$ lies in
		$
		[0.5/\sqrt{2},1]
		$
		thus we see that 
		\[
		|y-x|\in [1/6,1.5].
		\]
		We still have to perform the $\mod 1$ operation on each coordinate component of $x$ and $y$. Denote the following doubled set of $K$
		\[
		\tilde{K}=K\cup (K+(0,1))\cup (K+(1,0))\cup (K+(1,1)),
		\]
		then because $|y-x|\in [1/6,1.5]$, we can find $\tilde{y}, \tilde{x}\in \tilde{K}$ such that
		\[
		\tilde{y}-\tilde{x}=y-x.
		\]
		For each $k\in W_K$ we have seen that there is a pair of points $x,y\in\tilde{K}$ with $|x-y|\in [1/6,1.5]$ such that the direction vector $y-x$ has slope
		\[
		u3^{\{k\log 2/\log 3\}}.
		\]
		We denote the map $e:[0,1]\to S^1$ such that $e(t)$ is the direction vector in $S^1\subset\mathbb{R}^2$ with slope $u3^t.$ It is easy to see that this map is smooth; therefore it is bi-Lipschitz. Then we see that 
		\[
		e\left(\{k\log 2/\log 3\}_{k\in W_K}\right)\subset DD(K).
		\]
		However the dipole direction set $DD(K)$ has upper box dimension at most $2s<1$ and therefore its Lebesgue measure is $0$. By Lemma \ref{Target} $W_K$ must have upper Banach density $0$. For the second conclusion, let $N$ be a large integer and $a$ be an arbitrarily chosen integer. We notice that $\{k\log 2/\log 3\}_{k\in [a,a+N]\cap W_K}$ consists $r_N$-separated points for $r_N=N^{-13.3}$, see Lemma \ref{LM2}. Let $\epsilon>0$ be a small number then for all large enough $N$ we see that
		\[
		\#W_K\cap [a,a+N]\leq \left(\frac{1}{r_N} \right)^{2s+\epsilon}.
		\]
		If we choose $\epsilon$ to be small enough we see that
		\[
		\#W_K\cap [a,a+N]=O(N^{27s}).
		\]
	\end{proof}
	\section{Sinai's factor theorem: Casino with clocks}\label{SINAI}
	In this section we introduce Sinai's factor theorem and prove Theorem \ref{Cor}. The main technicalities here are similar to that in \cite[Section 6]{Wu} and our Theorem \ref{Cor} can be seen as a generalization of \cite[Theorem 6.1]{Wu}. 
	
	To have a intuitive idea in mind, consider a sequence of i.i.d random variables $\{X_n\}_{n\geq 1}$ with values in $\{0,1\}$ For any irrational number $\alpha$ we consider the sequence
	$
	\{X_nR^n_\alpha(0)\}_{n\geq 1}.
	$
	Intuitively, imagine a casino with a clock (which is unrealistic) with only one finger rotating with irrational angular speed ( $+\alpha\mod 1$ system). Whenever a gambler throws a coin with head up then he will check the clock. Then a sample path of the above random sequence would be a series of time a gambler observed. The results in this section can be intuitively stated as follows. For each gambler, almost surely, the time series he observed equidistributes in $[0,1]$, that is, the time series he observed does not depend on whether he is winning or losing. We shall discuss various different aspects towards the above intuition.
	\subsection{Bernoulli system}
	Let $\Lambda$ be a finite set of symbols and let $\Omega=\Lambda^{\mathbb{N}}$ be the space of one sided infinite sequences over $\Lambda.$ We define $S$ to be the shift operator, namely, for $\omega=\omega_1\omega_2\dots\in\Omega,$
	\[
	S(\omega)=\omega_2\omega_3\dots.
	\]
	Then we take a $\sigma$-algebra on $\Omega$ generated by cylinder subsets. A cylinder subset $Z\subset\Omega$ is such that
	$
	Z=\prod_{i\in\mathbb{N}}Z_i
	$
	and $Z_i=\Lambda$ for all but finitely many integers $i\in\mathbb{N}.$ We construct a probability measure $\mu$ on $\Omega$ by giving a probability measure $\mu_{\Lambda}=\{p_\lambda\}_{\lambda\in \Lambda}$ on $\Lambda$ and set $\mu=\mu^{\mathbb{N}}_{\Lambda}.$ We require here that $p_\lambda\neq 0$ for all $\lambda\in \Lambda$. Then this system is weak-mixing and has entropy $h(S,\mu)=\sum_{\lambda\in\Lambda} -p_\lambda\log p_\lambda.$ We call this system a Bernoulli system. We can also introduce a metric topology on $\Omega$ by defining $d(\omega,\omega')=\#\Lambda^{-\min\{i\in\mathbb{N}: \omega_i\neq\omega'_i\}}.$ This turns $\Omega$ into a compact and totally disconnected space. For $\omega\in\Omega$ and $r\in (0,1)$, we use $B(\omega,r)$ to denote the $r$-ball around $\omega$ with radius $r$ with respect to the metric $d$ constructed above.
	\subsection{Sinai's factor theorem}
	\begin{thm}[Sinai's factor theorem]\label{factor}
		Let $(X,S,\mu)$ be an ergodic dynamical system. Then any Bernoulli system $(\Omega,S_B,\nu)$ with $h(S_B,\nu)\leq h(S,\mu)$ is a factor of $(X,S,\mu).$ 
	\end{thm}
	
	Let $Ber=(\Omega,S,\mu)$ be a Bernoulli system on $\Omega=\Lambda^{\mathbb{N}}$. Let $\alpha\in (0,1)$ be an irrational number. Heuristically, the dynamical system $T$ looks like a stochastic process with a sequence of i.i.d random variables. For any $B\subset\Omega$ with $\mu(B)>0$ and $\omega\in \Omega$ the following set
	\[
	K(\omega,B)=\{k\in\mathbb{N}: S^k(\omega)\in B\}
	\]
	can be realized as randomly constructed by choosing each $k\in\mathbb{N}$ independently with probability $\mu(B)$. Then for any subset $K'\subset\mathbb{N}$ the chance that $K(\omega,B)\cap K'=\emptyset$ is $(1-\mu(B))^{\# K'}$ and it is small when $\# K'$ is large unless $\mu(B)=0$ which we assumed not to be the case. 
	\begin{defn}
		Let $(X,S,\mu)$ be a dynamical system and let $B\subset X$ be a subset. Then we can construct the following sequence
		\[
		K(x,B)=\{k\in\mathbb{N}: S^k(x)\in B\},
		\]
		and the following set for $\alpha\in [0,1),$
		\[
		A_{K(x,B)}(\alpha)=\overline{\{R^k_\alpha(0)\}}_{k\in K(x,B)}.
		\]
	\end{defn}
	
	\begin{lma}\label{BERANDROT2}
		Consider the Bernoulli system $(\Omega,S,\mu)$. Let $\{B_i\}_{i\in\mathcal{I}}$ be a finite pairwise disjoint family of measurable subsets of $\Omega.$ Suppose that $\sum_{i\in\mathcal{I}}\mu(B_i)\geq 1-\delta$ for a $\delta\in (0,1).$ Then there exists a set $\Omega'\subset\Omega$ with full $\mu$-measure such that for each $\omega\in\Omega$ and any integer sequence $K$ of lower natural density $\rho$ larger than $\delta$, there exists an $i=i(\omega,K)\in\mathcal{I}$ such that
		\[
		A_{K(\omega,B_i)\cap K}
		\]
		has  Lebesgue measure at least $\rho-\delta.$
	\end{lma}
	\begin{proof}
		For each $i\in\mathcal{I}$, $K(\omega,B_i)$ can be essentially viewed as a random sequence of integers obtained by deciding to choose each integer independently with probability $\mu(B_i).$ It is helpful to have this intuition in mind for what follows. We see that for almost all $\omega\in\Omega$, by the ergodicity of Bernoulli systems,
		\[
		d(K(\omega,B_i))=\mu(B_i),
		\]
		and the sequence of real numbers $\{R^k_\alpha(0)\}_{k\in K(\omega,B_i)}$ equidistributes in $[0,1]$ (we re-enumerate $K(\omega,B_i)$ with $\mathbb{N}$). This can be seen by considering the dynamical system $(\Omega\times [0,1],S\times R_\alpha, \mu\times \lambda)$ ($\lambda$ is the Lebesgue measure) which is ergodic because it is the product of a weakly mixing and a uniquely ergodic system, see also \cite[Lemma 6.5]{Wu}. Since $\mathcal{I}$ is a finite family, we see that for almost all $\omega\in \Omega$, for each $i\in\mathcal{I}$ the above results hold. We denote this full measure set as $\Omega'.$ We see that for each $\omega\in\Omega'$
		\[
		d(\cup_{i\in\mathcal{I}} K(\omega,B_i))\geq 1-\delta.
		\] 
		Now let $K$ be an arbitrarily chosen sequence with lower natural density $\rho>\delta,$ then we see that $K\cap(\cup_{i\in\mathcal{I}} K(\omega,B_i))$ has lower natural density at least $\rho-\delta>0.$ We denote for each $i\in\mathcal{I}$, $K_i=K\cap K(\omega,B_i)$ and $\rho_i=\overline{d}(K_i).$ Then we see that 
		\[
		\sum_{i\in\mathcal{I}}\rho_i\geq \rho-\delta.
		\]
		Now if $\rho_i<(\rho-\delta)\mu(B_i)$ for all $i\in\mathcal{I}$ we see that
		\[
		(\rho-\delta)\sum_{i\in\mathcal{I}}\mu(B_i)>\sum_{i\in\mathcal{I}}\rho_i\geq \rho-\delta.
		\]
		This implies that $\sum_{i\in\mathcal{I}}\mu(B_i)>1$ and it is impossible. So we see that there exists an $i\in\mathcal{I}$ such that $\rho_i\geq (\rho-\delta)\mu(B_i).$ Now we denote $\epsilon=\rho-\delta.$ For this $i$ we see that $K(\omega,B_i)\setminus K_i$ has lower natural density at most \[\mu(B_i)-\rho_i\leq (1-\epsilon)\mu(B_i).\] Now by the equidistribution property we see that for any interval $I\subset [0,1]$,
		\[
		K''=\{k\in K(\omega,B_i):R^k_{\alpha}(0)\in I   \}
		\]
		has natural density $\mu(B_i)|I|$ and therefore if $|I|\mu(B_i)> (1-\epsilon)\mu(B_i)$ then $K''$ has natural density strictly larger than $(1-\epsilon)\mu(B_i).$ Therefore $K_i\cap K''$ cannot be empty and thus we have $I\cap A_{K(\omega,B_i)\cap K}\neq\emptyset.$ This argument works for finite unions of intervals as well. For any finite collection of intervals with disjoint interiors $I_j,j\in\mathcal{J}$ with total length $\sum_{j\in\mathcal{J}} |I_j|>1-\epsilon$ we see that
		\[
		\left(\bigcup_{j\in\mathcal{J}} I_j\right)\cap A_{K(\omega,B_i)\cap K}\neq\emptyset.
		\]
		Then we see that $A^c_{K(\omega,B_i)\cap K}$ is open and has Lebesgue measure at most $1-\epsilon.$ This is because for any open set $O\subset [0,1]$, there exist a countable family $L_m,m\geq 1$ of intervals with disjoint interior such that $\sum_{m}|L_m|=\lambda(O),$ where $\lambda$ is the Lebesgue measure, see \cite[Theorem 1.3]{SS}. Then for any $\eta>0$ we can find a finite collection of those intervals with total length at least $\lambda(O)-\eta.$ We can apply this argument for $A^c_{K(\omega,B_i)\cap K}$ and for arbitrary small $\eta>0.$ As a result we see that $\lambda(A_{K\cap K(\omega,B_i)})\geq \epsilon$ as required.
	\end{proof}
	\begin{thm}\label{KK}
		Let $(X,S,\mu)$ be an ergodic dynamical system with $h(S,\mu)>0.$ We can find a Bernoulli factor $Ber=(\Omega,S_B,\nu)$ of $(X,S,\mu)$ with $h(S_B,\nu)=h(T,\mu)>0.$ Denote $f:X\to\Omega$ to be the factorization map. For a $\delta>0,$ let $B_i,i\in\mathcal{I}$ be a finite disjoint collection of measurable subsets in $\Omega$ with $\sum_{i\in\mathcal{I}}\nu(B_i)\geq 1-\delta.$ Then for $\mu$ almost all $x\in X,$ for any integer sequence $K$ with lower natural density $\rho>\delta$ there exist an $i\in\mathcal{I}$ such that
		\[
		A_{K\cap H_i(x)}(\alpha)\text{ has Lebesgue measure at least $\rho-\delta$,}
		\]
		where $
		H_i(x)=\{k\in\mathbb{N}: S^k(x)\in f^{-1}(B_i)\}.
		$
	\end{thm}
	\begin{proof}
		We can find a Bernoulli factor $Ber=(\Omega,S_B,\nu)$ of $(X,S,\mu)$ with $h(S_B,\nu)=h(S,\mu)>0$. The a straightforward application of Theorem \ref{KK} gives us the result.
	\end{proof}

	When $Ber$ is a factor of $(X,S,\mu)$ with the same entropy, then intuitively all the complicities are carried by $Ber$ and therefore the fibres of $f$ should not be too complicated with respect to the map $S$. The following result expresses this intuition in a clear way. The following result is known as Rohlin's disintegration theorem, and we adopt the version in \cite{S12}.
	\begin{defn}
		Let $f:X\to Y$ be a measurable map between two measurable spaces and let $\mu$ be a measure on $X$ with projection $\mu_Y=f\mu$ on $Y$. We call a collection of measures $\{\mu_y\}_{y\in Y}$ a system of conditional measures if the following properties hold,
		\begin{itemize}
			\item[1]: For all $y\in Y$, $\mu_y$ is a measure supported on $f^{-1}(y)$ and for $\mu_Y$ almost all $y\in Y$, $\mu_y$ is a probability measure.
			\item[2]: We have the law of measure disintegration. For all Borel set $B\subset X$, we have
			\[
			\mu(B)=\int \mu_{y}(B)d\mu_Y(y).
			\]
		\end{itemize}
		If $X,Y$ are also metric spaces ($f$ need not to be continuous) we require further that the following holds for $\mu_Y$ almost all $y\in Y$.
		\begin{itemize}
			\item[3]: $\mu_y=\lim_{r\to 0} \mu_{f^{-1}(B(y,r))},$ where the limit is in the weak* sense and $\mu_{f^{-1}(B(y,r))}$ is the conditional measure of $\mu$ on $f^{-1}(B(y,r))$, namely, for any Borel set $B\subset X$ with positive $\mu$ measure,
			\[
			\mu_{f^{-1}(B(y,r))}(B)=\frac{\mu(B\cap f^{-1}(B(y,r)))}{\mu(B)}.
			\]
		\end{itemize}
	\end{defn}
	\begin{thm}
		Let $f:X\to Y$ be a measurable map between two metric spaces with corresponding Borel $\sigma$-algebra. Then there exists a system of conditional measures.
	\end{thm}
	Then we have the following result due to \cite[Lemma 6.4]{Wu} which is a direct consequence of the conditional Shannon-McMillan-Breiman theorem, Egorov's theorem and the Portmanteau theorem.
	\begin{thm}[Wu]\label{Wu}
		Let $(X,S,\mu)$ be an ergodic dynamical system with $X$ being a Borel space. Let $\mathcal{A}$ be a finite partition of $X$ such that $\vee_{i=0}^\infty S^{-i}\mathcal{A}$ generates the sigma-algebra of $X.$ For each $x\in X$ not on the boundaries of sets in $\vee_{i=1}^{n} S^{-i}\mathcal{A}$,  for each $n\in\mathbb{N}$ we denote $A_n(x)$ the unique atom $A$ of $\vee_{i=0}^n S^{-i}\mathcal{A}$ such that $x\in A.$ 
		
		If $\mu$ does not give positive measures to boundaries of $S^{-i}\mathcal{A}$ for all $i\in\mathbb{N}$ and $h(S,\mu)>0$ then there exist a Bernoulli factor $(\Omega,S_B,\nu)$ with measurable factorization map $f:X\to \Omega$ and for each $\delta>0$ there exist a $X_\delta\subset X$ and a constant $C_\delta$ with the following properties,
		\begin{itemize}
			\item[1]:$\mu(X_\delta)>1-\delta.$
			\item[2]:For all $x\in X_\delta$ and $n\geq 1$, $\mu_{f(x)}(A_n(x))\geq C_\delta2^{-n\delta}$ and $\mu_{f(x)}$ is a probability measure.
			\item[3]:For all integers $n\geq 1$, there exists a measurable set $B^n_\delta\subset\Omega$ with $\nu(B^n_\delta)\geq 1-\delta$ and a $r=r(\delta,n)>0$ such that for all $\omega\in B^{n}_\delta$ and all atoms $A_n$ we have
			\[
			\frac{\mu(f^{-1}(B(\omega,r))\cap A_n)}{\mu(f^{-1}(B(\omega,r)))}\geq (1-\delta)\mu_{\omega}(A_n).
			\]
		\end{itemize}
	\end{thm}
	The following result is a generalization of \cite[Theorem 6.1]{Wu}.
	\begin{thm}\label{Cor}
		We adopt the conditions in Theorem \ref{Wu}. In addition we let $\epsilon>0$ be arbitrarily chosen in $(0,1)$ and $\alpha$ be an arbitrary irrational number in $(0,1).$ For each $\delta\in (0,1)$, there is a constant $c_\delta>0$ and $X'_\delta$ with full $\mu$ measure such that the following statement holds:
		\begin{itemize}
			\item[] For all $n\geq 1,$ all $x\in X'_\delta$ and all $K\subset\mathbb{N}$ with lower natural density at least $\rho>2\delta+\epsilon,$ there is a collection $\mathcal{M}_n=\mathcal{M}_n(x,K)$ of at most $c_\delta 2^{n\delta}$ atoms of $\vee_{i=0}^n S^{-i}\mathcal{A}$ with the following properties. Denote the union of elements in $\mathcal{M}_n$ as $M_n$. We construct the following sequence
			\[
			K'(x)=\{k\in\mathbb{N}: S^k(x)\in M_n \}.
			\]
			Then the following set has Lebesgue measure at least $\epsilon$
			\[
			A_{K\cap K'(x)}(\alpha)=\overline{\{R^k_\alpha(0): k\in K\cap K'(x) \}}.
			\]
			
		\end{itemize}

	\end{thm}
	\begin{proof}
		We use Theorem \ref{Wu} to find a set $X_\delta$ with $\mu(X_\delta)>1-\delta$. Then for each integer $n\geq 1$ we can find $B^n_\delta$ with $\nu(B^n_\delta)\geq 1-\delta$ and $r=r(\delta,n)>0.$ Without loss of generality we shall assume that $r=d^{-k}$ where $d$ is the number of digits of the Bernoulli system and $k$ is an integer. For each $\omega\in B^n_\delta$ we have
		\[
		\frac{\mu(f^{-1}(B(\omega,r))\cap A_n)}{\mu(f^{-1}(B(\omega,r)))}\geq (1-\delta)\mu_{\omega}(A_n).
		\]
		Now because of the topology we chose for $\Omega$, we see that $B(\omega,r)$ consists of all sequences in $\Omega$ with the same first $k$ digits as $\omega.$ In particular if $\omega'\in B(\omega,r)$ then $B(\omega',r)=B(\omega,r).$ This property reflects the fact that $\Omega$ is an ultrametric space. Notice that for any Bernoulli system $(\Omega,S_B,\nu)$, any ball of positive radius has positive $\nu$ measure. In particular $\mu(f^{-1}(B(\omega,r)))>0$ and by properties $(2)$ and $(3)$ in Theorem \ref{Wu}, for each $\omega'\in B^n_\delta\cap B(\omega,r)$ we have 
		\[
		\frac{\mu(f^{-1}(B(\omega,r))\cap A_n)}{\mu(f^{-1}(B(\omega,r)))}=\frac{\mu(f^{-1}(B(\omega',r))\cap A_n)}{\mu(f^{-1}(B(\omega',r)))}\geq (1-\delta)\mu_{\omega'}(A_n)\geq (1-\delta)C_{\delta} 2^{-n\delta}
		\]
		whenever $A_n=A_n(x)$ for some $x\in X_\delta\cap f^{-1}(\omega')$. Now it is possible to see that for all $x$ in the set $X_\delta\cap f^{-1}(B(\omega,r)\cap B^n_\delta)$ we have 
		\[
		\frac{\mu(f^{-1}(B(\omega,r))\cap A_n(x))}{\mu(f^{-1}(B(\omega,r)))}\geq (1-\delta)C_{\delta} 2^{-n\delta}.
		\]
		On the other hand we clearly have
		\[
		\sum_{\text{atoms }A_n}\frac{\mu(f^{-1}(B(\omega,r))\cap A_n)}{\mu(f^{-1}(B(\omega,r)))}=1.
		\]
		Since $\mu$ is not supported on boundaries of any atom we see that $X_\delta\cap f^{-1}(B(\omega,r)\cap B^n_\delta)$ can intersect at most
		\[
		\frac{2^{n\delta}}{(1-\delta)C_\delta}
		\]
		many atoms of  $\vee_{i=0}^n S^{-i}\mathcal{A}$ since different atoms can intersect only on boundaries. Now let $Y(\omega)=X_\delta\cap f^{-1}(B(\omega,r)\cap B^n_\delta)$. Since there are only finitely many $r$ balls in $\Omega$ we see that as $\omega$ varies in $B^n_\delta$ there are finitely many different sets of form $Y(\omega).$ Denote the collection of these sets as $\{Y_1,\dots, Y_{N(n)}\}$ where $N(n)$ is an integer. For each $i\in\mathcal{I}=\{1,\dots,N(n)\}$, let $\Omega(i)\subset B^n_\delta$ be the set of form $B(\omega,r)\cap B^n_\delta$ such that $Y_i=X_\delta\cap f^{-1}(\Omega(i)).$ We notice here that the union of all $Y_i$ is a rather large subset of $X$, more precisely we have the following result,
		\[
		\mu\left(\bigcup_{i\in\mathcal{I}}Y_i\right)= \mu\left(X_\delta\cap f^{-1}(B^n_\delta)\right)\geq 1-2\delta.
		\]
		For each $i\in \mathcal{I}$ we write the collection of atoms intersecting $Y_i$ as $\mathcal{M}_n(i)$ and write their union as $M_n(i).$ Then we saw that
		\[\#\mathcal{M}_n(i)\leq     \frac{2^{n\delta}}{(1-\delta)C_\delta}.\]
		Now we consider the following sequence for $x\in X,$
		\[
		K(x)=\{k\in\mathbb{N}: S^k(x)\in X_\delta \}
		\]
		by the ergodic theorem we see that for $\mu$ almost all $x\in X,$
		$
		K(x)
		$ 
		has natural density at least $1-\delta.$ For each $i\in\mathcal{I}$ and $x\in X$ we construct the following set
		\[
		K'_i(x)=\{k\in\mathbb{N}: S^k(x)\in M_n(i) \}
		\]
		and we see that
		\begin{eqnarray*}K(x)\cap K'_i(x)&=&\{k\in\mathbb{N}: S^k(x)\in M_n(i)\cap X_\delta\}\\
			&\supset& \{k\in\mathbb{N}: S^k(x)\in Y_i\}\\
			&=&\{k\in\mathbb{N}: S^k(x)\in f^{-1}(\Omega(i))\cap X_\delta\}\\
			&=&K(x)\cap K(f(x),\Omega(i)).\end{eqnarray*}
		By Lemma \ref{BERANDROT2} and Theorem \ref{KK} we see that for $\mu$ almost all $x\in X$ and any sequence $K$ with lower natural density at least $2\delta+\epsilon$ there exists an $i\in\mathcal{I}$ such that
		\[
		A_{K\cap K(x)\cap K(f(x),\Omega(i))}
		\]
		has Lebesgue measure at least $\epsilon.$ This is because $K\cap K(x)$ has lower natural density at least $\delta+\epsilon$ for $\mu$ almost all $x\in X$ and $\sum_{i\in\mathcal{I}}\nu(\Omega(i))=\mu(B^n_\delta)\geq 1-\delta$. This theorem follows since the above argument holds for all $n\geq 1$ and we can find a full measure set $X'_\delta\subset X$ which satisfies all our requirements.
	\end{proof}
	\section{Large sets, Bernoulli factors}\label{LARGE}
	We now deal with the case when $\Haus A_2+\Haus A_3\in (1/2,1).$
	\begin{proof}[Proof of Theorem \ref{THMONE}]
		For the moment let $A\subset\mathbb{R}^2$ be an arbitrary compact set. We define the following function $g_A:\mathbb{R}^2\times [0,1]\to \{0,1\}.$ For $(a,t)\in \mathbb{R}^2\times [0,1],$ we assign the value
		\[
		g_A(a,t)=1 \text{ if and only if }  ([a+0.5v_t,a+v_t]\cup [a-0.5v_t,a-v_t])\cap A\neq\emptyset,
		\]
		where we use $[x,y]$ with $x,y\in\mathbb{R}^2$ for the line segment from $x$ to $y$ and $v_t$ is the vector with slope $3^t$ whose $Y$-projection has length $1$. To see that $g_A$ is measurable it is enough to see that $\{g(a,t)=0\}$ is Borel measurable. Let $(a,t)\in \mathbb{R}^2\times [0,1]$ be such that $([a+0.5v_t,a+v_t]\cup [a-0.5v_t,a-v_t])\cap A=\emptyset.$ Since $A$ is compact, for each $\eta\in [0,5,1]\cup [-1,-0.5]$ we see that $a+\eta v_t\notin A$ and therefore there exists positive number $r(\eta)>0$ such that $B(a+\eta v_t,r(\eta))\cap A=\emptyset.$ We know that the segment $([a+0.5v_t,a+v_t]\cup [a-0.5v_t,a-v_t])$ is compact, therefore there exist positive number $r>0$ such that $([a+0.5v_t,a+v_t]\cup [a-0.5v_t,a-v_t])^r\cap A=\emptyset,$ where $([a+0.5v_t,a+v_t]\cup [a-0.5v_t,a-v_t])^r$ is the $r$-neighbourhood of $[a+0.5v_t,a+v_t]\cup [a-0.5v_t,a-v_t].$ Then it is easy to see that there exist two positive numbers $r(a),r(t)$ such that for each $(a',t')\in \mathbb{R}^2\times [0,1]$ with $|a'-a|<r(a)$ and $|t'-t|<r(t)$ we have
		\[
		([a+0.5v_{t'},a+v_{t'}]\cup [a-0.5v_{t'},a-v_{t'}])\cap A=\emptyset.
		\]
		This shows that $\{g_A(a,t)=0\}$ is in fact an open set and therefore $g_A$ is measurable.
		
		Now let $A=(A_3\times A_2)\cup (A_3\times A_2+(0,\pm 1))\cup (A_3\times A_2+(\pm 1,0))\cup (A_3\times A_2+(\pm 1,\pm 1))$ (there are in total $9$ translated copies of $A_3\times A_2$) and $\alpha=\log 2/\log 3$. In what follows we omit the subscript $A$ in $g_A.$ 
		
		Suppose the uniform sparseness condition does not hold. We shall restrict to the special case when $u\in [1,3)$ and the general case follows similarly. We see that there is a positive number $\rho$, a sequence $\{l_k\}_{k\geq 1}$ of lines with slope $\{u_k\}_{k\geq 1}\subset [1,3)$, a sequence $\{a_k\}_{k\geq 1}\subset A_3\times A_2$ of points with $a_k\in l_k, k\geq 1$ and sequences of integers $N_k,n_k$ with $N_k\to\infty$ such that
		\[
		\frac{1}{N_k}\#W(l_k,a_k)\cap [n_k+1,n_k+N_k]\geq \rho>0.
		\]
		For convenience, we have written $W(l_k,a_k)$ for $W(l_k\cap A,a_k).$ Since we are always considering the intersection, it is possible to drop the $\cap A$ without causing confusions.  Denote $t_k=\log u_k/\log 3\in [0,1).$ Now we define a dynamical system $(U,S,\mu)$ according to this initial choice of $a_k,l_k$. First we set $U=[0,1]^2\times [0,1].$ Let $a=(x,y)\in A_3\times A_2$ and $\theta\in [0,1)$
		\[
		S(a,\theta)=(T((x,y),\theta),R_{\alpha}(\theta)),
		\]
		\[
		T((x,y),\theta)=
		\begin{cases*}
		(x,2y \mod 1), \text{ if $\theta+\alpha\leq 1$}\\
		(3x\mod 1,2y\mod 1), \text{ if $\theta+\alpha>1$}.
		\end{cases*}
		\]
		We notice that for any $a\in A$ and any $t\in [0,1]$ the orbit of $(a,t)$ always lies in $(A_3\times A_2)\times [0,1].$   Having defined the dynamics we now construct a measure.
		Denote $x_k=(a_k,t_k).$ For each $k$ we construct the following measure in $\mathcal{P}(U).$ 
		\[
		\mu_k=\frac{1}{N_k}\sum_{i=n_k+1}^{n_k+N_k} \delta_{S^i(x_k)}.
		\]
		Then by taking a sub-sequence if necessary we assume that 
		\[
		\mu_k\to \mu
		\]
		weakly in $\mathcal{P}(U).$ This measure $\mu$ is not necessarily $S$-invariant because it might give positive measure on the discontinuities of $S$. If we identify  $[0,1]^2$ with $\mathbb{R}^2/\mathbb{Z}^2=\mathbb{T}^2$ then $S$ is discontinuous precisely at points $(a',t')$ with $t'=1-\alpha.$ This is where we are about to choose a different multiplication map for the $[0,1]^2$ component. However it is easy to see that the projection of $\mu$ onto the $[0,1]$ component is precisely the Lebesgue measure because $\alpha\notin \mathbb{Q}$ and $R_\alpha$ is uniquely ergodic. Thus $S$ is $\mu$-a.e. continuous and therefore $\mu$ is $S$-invariant (see Theorem \ref{PORT}, statement 3). Now we take a $\mu$-typical $x\in U$. Suppose that $x=(a',t')$, we want to estimate the following average,
		\[
		\lim_{N\to\infty}\frac{1}{N}\sum_{i=0}^{N-1} g(S^{i}(a',t')).
		\] 
		Thus for $\mu$.a.e $(a',t')$  we denote $\mu_{a',t'}$ to be the ergodic component of $(a',t')$ and we see that,
		\[
		\lim_{N\to\infty}\frac{1}{N}\sum_{i=0}^{N-1} g(S^{i}(a',t'))=\int g(a,t)\mu_{a',t'}(a,t).
		\]
		Suppose $\sigma(a',t')$ is the ergodic disintegration measure of $\mu$ against the $S$-invariant $\sigma$-algebra we see that
		\begin{eqnarray*}
			\int\int g(a,t)\mu_{a',t'}(a,t)\sigma(a',t')&=&\int g(a,t)\mu(a,t)\\
			&\geq&\limsup_{k\to\infty} \frac{1}{N_k}\sum_{i=n_k+1}^{n_k+N_k} g(S^i(x_k))\\
			&\geq&\limsup_{k\to\infty} \frac{1}{N_k}\#W(l_k,a_k)\cap [n_k+1,n_k+N_k]\geq \rho>0.
		\end{eqnarray*}
		In the second step, we have used the fact that $\{g(a,t)=1\}$ is a closed set and we also used the Portmanteau theorem ( Theorem \ref{PORT}, statement 2). For the third step we used the fact that $A_3$ is $\times 3\mod 1$ invariant and $A_2$ is $\times 2 \mod 1$ invariant. We would get equality in the third step if the sets $A_2,A_3$ would be strictly invariant under the maps $\times 2,\times 3$ respectively. Intuitively we transferred the upper Banach density in our initial data to the upper natural density almost surely along the orbit average. For this reason, for each $(a,t)\in U$ we denote $W'(a,t)$ to be the following sequence,
		\[
		W'(a,t)=\{k\in\mathbb{N}: g(S^{k}(a,t))=1 \}.
		\]
		We see that there is an ergodic component $\mu_{a',t'}$ such that
		\[
		\int g(a,t)\mu_{a',t'}(a,t)\geq \rho.
		\]
		Consider now the dynamical system $(U,S,\mu_{a',t'})$, it is ergodic by construction with the property that for $\mu_{a',t'}$ almost all $(a'',t'')\in U$
		\[
		\lim_{N\to\infty}\frac{1}{N}\sum_{i=0}^{N-1} g(S^{i}(a'',t''))\geq \rho.
		\] 
		In order to apply Theorem \ref{Cor} we need to address some issues. We divide the rest of proof into three subsections.
		\subsection{Partitions and boundaries}
		First we take an initial partition $\mathcal{A}$ of $[0,1]^2\times [0,1]$ by taking
		\[
		A_{i,j}\times C_k,
		\]
		where we define for $(i,j)\in \{0,1,2\}\times\{0,1\}$ the following partition $\mathcal{B}$ of $[0,1]^2$
		\[
		A_{i,j}=\left[\frac{i}{3},\frac{i+1}{3}\right]\times \left[\frac{j}{2},\frac{j+1}{2}\right],
		\]
		and a partition $\mathcal{C}$ of $[0,1]$ as $C_1=[1-\alpha,1], C_2=[0,1-\alpha).$ We see that $\vee_{i=0}^\infty S^{-i} \mathcal{A}$ generates the Borel $\sigma$-algebra of $[0,1]^2\times S^1$ and therefore $h(S,\mu_{a',t'})=h(S,\mu_{a',t'},\mathcal{A}).$ Our first issue is that $\mu_{a',t'}$ could give positive measure on boundaries of $\{S^{-i} \mathcal{A}\}_{i\geq 0}$. We see that the $[0,1]$ component of $\mu_{a',t'}$ is $+\alpha \mod 1$ invariant and thus it is the Lebesgue measure. If $\mu_{a',t'}$ does give positive measure on boundaries of $\{S^{-i} \mathcal{A}\}_{i\geq 0}$ then its $[0,1]^2$ component gives positive measure on boundaries of the $[0,1]^2$ component of $\{S^{-i} \mathcal{A}\}_{i\geq 0}$, which are rectangles with edges that project to either dyadic rational numbers on $Y$-axis or triadic rational numbers on $X$-axis. In this case the $Y$-component of $\mu_{a',t'}$ is then supported on finitely many rational numbers since it is $\times 2\mod 1$ invariant and we can focus on the $X$-component. For the other case, the projection on $X$-axis does not define a dynamical system. In this case it can be seen that the $[0,1]^2$ component of $\mu_{a',t'}$ supports on finitely many horizontal lines with rational $X$-coordinates. Suppose the former case and the later case can be treated in a similar way. We consider the following dynamical system
		\[
		(A_3\times [0,1], S_X,\mu^X_{a',t'}).
		\]
		Here $S_X$ is defined as follows,
		\[
		S_X(x,\theta)=(T_X(x,\theta),R_{\alpha}(\theta)),
		\]
		\[
		T_X(x,\theta)=
		\begin{cases*}
		x, \text{ if $\theta+\alpha\leq 1$}\\
		3x\mod 1, \text{ if $\theta+\alpha>1$},
		\end{cases*}
		\]
		and $\mu^X_{a',t'}$ is the corresponding projected measure. If $\mu^X_{a',t'}$ still supports on boundaries we see that the $[0,1]^2$ component of $\mu_{a',t'}$ supports on finitely many rational points and this case the result is obvious. Therefore we can assume that at least one of the $X$ or $Y$ coordinate projections of $\mu_{a',t'}$ does not support on boundaries and we then perform the following entropy arguments for either $(U,S,\mu_{a',t'})$ or one of its projections. We only illustrate the argument for $(U,S,\mu_{a',t'})$ and the arguments for its projections are similar.
		\subsection{Zero entropy}
		We now consider the case when $h(S,\mu_{a',t'})=0$. In this case for each integer $n\geq 1$, the atoms of $\vee_{i=0}^n S^{-i} \mathcal{A}$ are of form $B\times C$ where $B\subset [0,1]^2$ is a rectangle with dimension $3^{-n'}\times 2^{n}$ where $n'$ satisfies $2^{-n}\leq 3^{-n'}\leq 3\times 2^{-n}$ (so the rectangle is almost a square) and $C$ is one of the atoms of $\vee_{i=0}^n R^{-i}_{\alpha} \mathcal{C}.$ The number of atoms in $\vee_{i=0}^n R^{-i}_{\alpha} \mathcal{C}$ is at most $2n$ and for each $C$ the number of different atoms $B\times C$ is between $2^{2n}/3$ and $3\times 2^{2n}.$ Now if the entropy $\frac{1}{n}H(\mu_{a',t'},\vee_{i=0}^n S^{-i} \mathcal{A})$ is smaller than a given small number $\epsilon$ for all large enough $n$ then there exist $\delta(\epsilon)=O(\epsilon)$ such that $O(2^{\delta n})$ many atoms in $\vee_{i=0}^n S^{-i} \mathcal{A}$ support at least $1-\delta$ portion of $\mu_{a',t'}$ measure. 
		
		To see this, let $V$ be a finite set of points of cardinality greater than $2^n$ and for each $v\in V$ we give a probability $p_v\in (0,1)$ such that $\sum_{v\in V}p_v=1.$ If the entropy, namely $-\sum_{v\in V} p_v\log p_v<n\epsilon'$ for a number $\epsilon'>0$, then for another number $\delta'>0$ we define the following subset
		\[
		V_{\delta'}=\{v\in V: p_v\geq 2^{-n\delta'}\}.
		\]
		Then it is easy to see that
		\[
		\sum_{v\in V}p_v(-\log p_v)=\sum_{v\in V_{\delta'}}p_v(-\log p_v)+\sum_{v\notin V_{\delta'}}p_v(-\log p_v)<n\epsilon'.
		\]
		If $v\notin V_{\delta'}$ then $-\log p_v>n\delta'\log 2$ and therefore
		\[
		n\epsilon'>\sum_{v\notin V_{\delta'}}p_v(-\log p_v)> \sum_{v\notin V_{\delta'}} n\delta' p_v\log 2.
		\]
		Then we see that
		\[
		\sum_{v\in V_{\delta'}}p_v=1-\sum_{v\notin V_{\delta'}}p_v>1-\frac{1}{\log 2}\frac{\epsilon'}{\delta'}.
		\]
		On the other hand because $\sum_{v\in V}p_v=1$ we see that
		\[
		\#V_{\delta'}\leq 2^{n\delta'}.
		\]
		Then we choose $\delta'=\sqrt{\epsilon'}$ and we see that at least $1-\delta'/\log 2$ portion of the measure $p_v,v\in V$ is supported in a collection of less than $2^{n\delta'}$ many points in $V$.
		
		Now we apply the above result to our dynamical system $(U,S,\mu_{a',t'})$ and denote the collection of those $O(2^{\delta n})$ atoms in $\vee_{i=0}^n S^{-i} \mathcal{A}$ as $\mathcal{M}_n$ and their union as $M_n$. We can choose an $\epsilon>0$ such that $\delta=\delta(\epsilon)<\rho.$ Now because $\mu_{a',t'}$ is an ergodic component of $a',t'$ we see that by the ergodic theorem, for $\mu_{a',t'}$.a.e $(a'',t'')\in U,$
		\[
		\{k\in\mathbb{N}: S^k(a'',t'')\in M_n \}
		\]
		has natural density at least $1-\delta$. Since $1-\delta+\rho>1$ we see that
		\[
		K=\{k\in\mathbb{N}: S^k(a'',t'')\in M_n \}\cap W'(a'',t'')
		\]
		has natural density at least $\rho-\delta.$ For each $k\in K$ we see that $S^k(a'',t'')\in M_n$ and $g(S^k(a'',t''))=1.$ Denote $a''_k,t''_k$ to be the $[0,1]^2$ and $[0,1]$ components of $S^k(a'',t'')$ respectively. Then we can find a point $b''_k\in A$ such that $|\pi_Y(a''_k-b''_k)|\in [0.5,1]$ and the line segment $[a''_k,b''_k]$ has slope $3^{t''_k}.$ This implies that $|a''_k-b''_k|\in [0.5,\sqrt{2}]\subset [1/6,1.5].$ 
		
		As $\mathcal{M}_n$ is a collection of at most $O(2^{\delta n})$ many atoms in $\vee_{i=0}^{n}S^{-i}\mathcal{A},$ the $[0,1]^2$ component of $\mathcal{M}_n$ consist at most $O(2^{\delta n})$ many almost squares, notices that they are not necessary disjoint. We denote the $[0,1]^2$ component of $\mathcal{M}_n$ as $\mathcal{Q}_n$ and we see that for each $k\in K$, there exist a $Q\in\mathcal{Q}_n$ such that $a''_k\in Q.$ It is also easy to see that $t''_k=R^{k}_{\alpha}(t'').$ Let $t^*$ be a limit point of $\{t''_k\}_{k\in K}.$ Since $\mathcal{Q}_n$ is a finite collection of closed sets we see that $\bigcup_{Q\in\mathcal{Q}_n}Q$ s a closed set. Assume that $\lim_{j\to\infty} t''_{k_j}=t^*$ for a subsequence $\{k_j\}_{j\in\mathbb{N}}$ of $K.$ Then $\{a''_{k_j}\}_{j\in\mathbb{N}}$ has a limit point $a^*$ in $\bigcup_{Q\in\mathcal{Q}_n}Q.$ Since $A_3,A_2$ are compact we see that this limit point is contained in $A_3\times A_2$ as well. Moreover since $A$ is also compact we see that we can assume the sequence $\{b_{k_j}\}_{j\in\mathbb{N}}$ converges to a limit point $b^*$ in $A$ and $|b^*-a^*|\in [1/6,1.5]$ and the line segment $[a^*,b^*]$ has slope $3^{t^*}.$ We have seen that $A_K=\overline{\{R^k_\alpha(t''): k\in K}\}$ has Lebesgue measure at least $\rho-\delta,$ see Lemma \ref{Target}. Recall the smooth map $e:[0,1]\to S^1$ constructed in the proof of Theorem \ref{THMHALF}. Then it is easy to see that there exists constant $c>0$ such that we can find a set $E$ of directions with Lebesgue measure at least $c(\rho-\delta)$ such that for each $e'\in E$ there is a point $x_{e'}\in A_3\times A_2$ and $Q\in\mathcal{Q}_n$ such that $x_{e'}\in Q.$ Moreover there exists $y_{e'}\in A$ with distance $|x_{e'}-y_{e'}|\in [1/6,1.5]$ and 
		\[
		\frac{y_{e'}-x_{e'}}{|y_{e'}-x_{e'}|}=e'.
		\] 
		We notice that $A$ and $A_3\times A_2$ have the same Hausdorff dimension. For all large enough integers $n$, we can find at least $0.5c(\rho-\delta) 2^{n}$ many $2^{-n}$-separated directions in $E$ and we denote this collection of direction as $\mathcal{E}_n.$ By the pigeonhole principle we see that there exists $Q\in\mathcal{Q}_n$ such that it contains $O(2^{-\delta n}(0.5c(\rho-\delta) 2^{n}))$ many points of form $\{x_{e'}\}_{e'\in \mathcal{E}_n}.$ Then the corresponding points $y_{e'}$ are all at least $0.5/2^n$-separated. As this holds for all large enough $n$ we see that this implies that $\ubox A\geq 1-\delta$ but we constructed $\delta=O(\epsilon)$ therefore by letting $\epsilon$ be small enough we obtain a contradiction because we assumed that $\Haus A_3\times A_2=\ubox A_3\times A_2<1$.
		\subsection{Positive entropy}
		Now finally we can assume that $(U,S,\mu_{a',t'})$ has positive entropy, that is, $h(S,\mu_{a',t'})>0.$ We saw that for $\mu_{a',t'}$ almost all $x=(a'',t'')\in U$, $W'(a'',t'')$ has lower natural density at least $\rho.$ Now we want to apply Theorem \ref{Cor}. Let $\delta>0$ be such that $\rho>2\delta.$ Then exists a constant $c_\delta>0$ and for each $n\geq 1$ there exist a set $U_\delta\subset U$ with full $\mu_{a',t'}$ measure such that for each $x\in U_\delta$, there is a collection $\mathcal{M}_n$ of at most $c_\delta 2^{\delta n}$ atoms in $\vee_{i=0}^n S^{-i} \mathcal{A}$ with union $M_n$ such that 
		\[
		A_{W'(a'',t'')\cap \{k\in\mathbb{N}: S^k(x)\in M_n\}}
		\]
		has Lebesgue measure at least $\rho-2\delta.$ Then the rest of the argument is the same as that of the zero entropy case. 
	\end{proof}
	\section{Further remarks and problems}
	\subsection{Casinos with multidimensional clocks}\label{HIGH}
	In this paper we only considered problems related to intersections between two invariant sets. One reason is that in Theorem \ref{Cor} we coupled a Bernoulli system with an irrational rotation on the unit circle. There is no problem if we replace the irrational rotation with an irrational torus rotation. Let $\mathbb{T}^k$ be the unit torus. We view it as $[0,1]^k.$ Suppose that $\alpha_1,\dots,\alpha_k$ are irrational numbers which are linearly independent over the field of rational numbers. Then the action
	\[
	(x_1,\dots,x_k)\to (x_1+\alpha_1\mod 1,\dots,x_k+\alpha_k\mod 1)
	\]
	is an irrational torus rotation. Like its one dimensional brother, irrational torus rotations are uniquely ergodic with the Lebesgue measure. One can also study discrepancy estimates, see \cite{DT97}. All results in Section \ref{SINAI} can be generalized in this way. Let $p_1,\dots,p_k$ be $k\geq 2$ integers such that $1, \log p_1/\log p_2, \dots, \log p_1/\log p_k$ are linearly independent over the field of rational numbers. We can consider $l\cap A_{p_1}\times\dots\times A_{p_k}$ with a line $l$ in $\mathbb{R}^k$ which is not parallel with the coordinate axis. We also assume that $l$ is not contained in any subspaces generated by coordinate axis, otherwise we can drop some of $A_{p_1},\dots,A_{p_k}$. To see how to obtain a torus rotation, let $(x_1,\dots,x_k,\theta_2,\dots,\theta_k)\in [0,1]^{2k-1},$ we define the following map (which can be viewed as a higher dimensional version of the map $T$ defined in the proof of Theorem \ref{THMONE})
	\[
	T(x_1,\dots,x_k,\theta_2,\dots,\theta_k)=(y_1,\dots,y_k,\{\theta_2+\log p_2/\log p_1\},\dots,\{\theta_k+\log p_k/\log p_1\})
	\]  
	where $y_1,\dots,y_k$ are determined as follows
	\[
	y_1=\{p_1 x_1\}
	\]
	and for each $i\in\{2,\dots,k\},$
	\[
	y_i=\begin{cases}
	x_i & \text{ if $\theta_i-\log p_1>0$}\\
	\{p_i x_i\} & \text{ else}.
	\end{cases}
	\]
	Now we allow the direction vector of $l$ range inside $S^{k-1}$ whose coordinate components are contained in $(\delta,1-\delta)$ where $\delta>0$ can be arbitrarily chosen. Then if $l\cap A_{p_1}\times\dots\times A_{p_k}$ is large (in terms of sparseness which can be defined similarly for lines in $\mathbb{R}^k$), by using the torus rotation with vector $(\log p_1/\log p_2,\dots,\log p_1/\log p_k)$ we see that $A_{p_1}\times\dots\times A_{p_k}$ would have dimension at least $k-1$, the dimension of $S^{k-1}.$ Therefore we can upgrade Theorem \ref{THMONE} for intersections among more than two sets. As the main technical steps are the same for all $k\geq 2,$ we only illustrated the proof for $k=2$ in which case we have a better visualization. To be precise, we state the following higher dimensional version of Corollary \ref{UniBox}.
	\begin{cor}
		Let $k\geq 2$ be an integer. Let $A_{p_1},\dots, A_{p_k}$ be $k$ closed invariant subsets of $[0,1]$ with respect to $\times p_1 \mod 1, \times p_2 \mod 1\dots$ respectively. Assume that $\log p_1/\log p_i$ for $i\in\{2,\dots,k\}$ are irrational numbers which are linearly independent over the field of rational numbers. Suppose that $\sum_{i=1}^k \Haus A_{p_i}<k-1$ then for each $2k$-tuple $u_1,\dots, u_k, v_1,\dots,v_k$ of non-zero real numbers we have
		\[
		\ubox \cap_{i=1}^k (u_i A_{p_i}+v_i)=0.
		\]  
		Moreover, let $\delta>0$ be an arbitrarily chosen positive number and suppose that $\delta<|u_i|<\delta^{-1}$ for each $i\in\{1,\dots,k\}.$ Then for each $\epsilon>0,$ there is an integer $N_\epsilon>0$ such that
		\[
		N(\cap_{i=1}^k (u_i A_{p_i}+v_i),2^{-N})\leq N^{\epsilon}
		\]
		for all $N\geq N_\epsilon.$ The choice of $N_\epsilon$ does not depend on $u_i,v_i.$ 
	\end{cor}
	There is one important point to note. For $k\geq 3,$ it is surprisingly not an easy task to produce even one example of integers $p_1,\dots,p_k$ to satisfy the condition that
	\begin{align*}
	1,\frac{\log p_1}{\log p_2},\dots,\frac{\log p_1}{\log p_k}\tag{*}
	\end{align*}
	are $\mathbb{Q}$-linearly independent. Problems of this kind are related to the study of algebraic relations among logarithms of algebraic numbers.  We now show that a conjecture of Schanuel (see below) implies the above $\mathbb{Q}$-linear independence as long as $p_1,\dots,p_k$ are multiplicatively independent, i.e.,
	\[
	1,\frac{\log p_2}{\log p_1},\dots,\frac{\log p_k}{\log p_1}
	\]
	are $\mathbb{Q}$-linearly independent, for example, $p_1=2,p_2=3,p_3=5.$
	
	Let $P_k(x_1,\dots,x_k)$ be the symmetric polynomial of degree $(1,\dots,1)$ (there are $k-1$ many ones). For example, we have
	\[
	P_3(x_1,x_2,x_3)=x_1x_2+x_1x_3+x_2x_3.
	\]
	For the above $\mathbb{Q}$-linear independence of (*) we would need that $(\log p_1,\dots,\log p_k)$  does not solve the polynomial equation
	\[
	P_k(n_1x_1,\dots,n_kx_k)=0.
	\]
	unless $n_1,\dots,n_k=0.$ We recall a conjecture of Schanuel,see \cite{A71}.
	\begin{conj}[Schanuel's conjecture]
		Let $z_1,\dots,z_k$ be $k\geq 1$ many $\mathbb{Q}$-linearly independent complex numbers. Then the transcendence degree of $(z_1,\dots,z_k,e^{z_1},\dots,e^{z_k})$ is at least $k.$
	\end{conj}
	We replace $z_i=\log p_i$ in the above conjecture. As a result the conjecture says that \[(\log p_1,\dots,\log p_k,p_1,\dots,p_k)\] has transcendence degree at least $k.$ As $p_1,\dots,p_k$ are already integers, this implies that $(\log p_1,\dots,\log p_k)$  transcendence degree $k$, i.e. $\log p_1,\dots,\log p_k$ are algebraically independent. Thus $$P_k(n_1\log p_1,\dots,n_k\log p_k)=0$$ implies that $n_1,\dots,n_k=0.$
	\section{Acknowledgement}
	HY was financially supported by the University of St Andrews. HY was also financially supported by the University of Cambridge and the Corpus Christi College, Cambridge. HY has received funding from the European Research Council (ERC) under the European Union’s Horizon 2020 research and innovation programme (grant agreement No. 803711). HY thanks the anonymous referees for help comments.
	\providecommand{\bysame}{\leavevmode\hbox to3em{\hrulefill}\thinspace}
	\providecommand{\MR}{\relax\ifhmode\unskip\space\fi MR }
	\providecommand{\MRhref}[2]{%
		\href{http://www.ams.org/mathscinet-getitem?mr=#1}{#2}
	}
	\providecommand{\href}[2]{#2}

\end{document}